\documentclass{amsart}
\usepackage{amscd,amsmath,amssymb,amsfonts,verbatim, xcolor,graphicx, calrsfs}
\usepackage{tikz-cd}
\usepackage[all]{xy}
\allowdisplaybreaks
\DeclareMathAlphabet{\pazocal}{OMS}{zplm}{m}{n}
\makeatletter
\newcommand*\dotp{\mathpalette\dotp@{.5}}
\newcommand*\dotp@[2]{\mathbin{\vcenter{\hbox{\schttps://www.overleaf.com/project/5b8ed5c3505e540a5408998ealebox{#2}{$\m@th#1\bullet$}}}}} 
\makeatother

\newtheorem{theorem}{Theorem}[subsection]
\newtheorem{corollary}[theorem]{Corollary}

\newtheorem{lemma}[theorem]{Lemma}  
\newtheorem{proposition}[theorem]{Proposition}
\theoremstyle{definition}
\newtheorem{definition}[theorem]{Definition}
\newtheorem*{conjecture*}{Conjecture}
\newtheorem*{remark}{Remark}

\numberwithin{equation}{subsection}
\title{ T\MakeLowercase{he} C\MakeLowercase{how}-K\MakeLowercase{ontsevich dilogarithm } }
\author{S\MakeLowercase{inan} \"{U}\MakeLowercase{nver}}
\address{Ko\c{c} University, Mathematics Department. Rumelifeneri Yolu, 34450, Istanbul, Turkey}
\email{sunver@ku.edu.tr}

 \subjclass[2010]{19E15, 14C25}
\begin{document}
\maketitle
\noindent

\begin{abstract}
Based on a variant of the Kontsevich $1\frac{1}{2}$-logarithm function, we construct a regulator for a curve over the ring of dual numbers of a field of characteristic $p.$ This   also leads to an infinitesimal invariant of certain cycles in characteristic $p.$  
\end{abstract}

\section{Introduction}

The Shannon entropy function defined by 
$$
H(x):=-x
\log (x)-(1-x)  \log (1-x),
$$ 
for $0<x<1,$
satisfies the fundamental equation of information theory: 
\begin{align}\label{entropy-eq}
  H(x) + (1 - x)H (\frac{y}{1-x}) = H(y) + (1 - y)H (\frac{x}{1-y}).  
\end{align}
This same equation reappeared in the work of Cathelineau (\cite{cat-hom}, \cite{cat2011}) which gave an infinitesimal analog of Hilbert's third problem. If $k$ is a field of characteristic 0, let $\beta_{2}(k)$ be the vector space over $k$ generated by the symbols $\langle x \rangle ,$ for $x \in k^{\times} \setminus \{ 1 \}$ with relations generated by 
\begin{align*}
    \langle x \rangle -\langle y \rangle+(1-x)\langle \frac{y}{1-x}\rangle - (1-y)\langle  \frac{x}{1-y} \rangle
\end{align*}
when $x + y \neq 1.$ The map $D: \beta_{2}(k)\to k\otimes k^{\times}$ defined on the generators by
\begin{align*}
    D(\langle a \rangle):=a \otimes a +(1-a) \otimes (1-a)
\end{align*}
%$D(\langle a \rangle):=a \otimes a +(1-a) \otimes (1-a)$
is an infinitesimal analog of the Dehn invariant.  Cathelineau proves  that the  cokernel of $D$ is isomorphic to $\Omega^{1}_{k}$ \cite[Theor\`eme 1]{cat-hom}. The infinitesimal version of the scissors congruence group can be seen as the limit  of the hyperbolic scissors congruence group as it approached to the euclidean one (\cite{vol},\cite{euc}). The observation that in the limit the hyperbolic volume map should approach the euclidean volume map, led Bloch and Esnault to define an additive dilogarithm based on a $K$-theoretic complex in \cite{b-e}. A similar dilogarithm map was expected  on an infinitesimal scissors congruence group \cite{euc}.  Letting $k_m:=k[t]/(t^m),$ 
such a dilogarithm function  $\ell i_{2}$ (\textsection 2.2.1), on the Bloch group $B_{2}(k_m)$ (\textsection 2.1)  is defined in \cite{un-add}. Letting  $\delta: B_{2}(k_m)\to \Lambda ^{2}k_m ^{\times }$  be the Bloch complex (\textsection 2.1), it was shown that the infinitesimal part of   $ker(\delta)$ is isomorphic to the indecomposable part of $K_3(k_{m})/K_{3}(k),$  and the infinitesimal part of $coker(\delta)$ is isomorphic to $K_{2} ^{M}(k_{m})/K_{2} ^{M}(k),$ as expected.   
Moreover,  for $m=2,$ the Bloch complex includes Cathelineau's complex $\beta_{2}(k) \to k \otimes k^{\times}$ as a subcomplex.

 The indecomposable part of $K_{3}(k_{2})/K_{3}(k)$ is isomorphic to $k\oplus k,$ if  $k$ is a field of characteristic $p\geq 5,$ whereas it is isomorphic to $k,$ if $k$ is of  characteristic 0. This suggests that in characteristic $p,$ there might be  two independent  dilogarithm functions on $B_{2}(k_2).$  In \cite{kont}, Kontsevich defined the function  
$$
 \pounds_{1}(s):=\sum_{1 \leq i \leq p-1}\frac{s^{i}}{i},
$$
 which he called the $1\frac{1}{2}$-logarithm since it satisfies the four-term functional equation  (\ref{entropy-eq}).  
This function was modified in \cite{un-ao}, to define a characteristic $p$ dilogarithm  map $\ell i_{2} ^{(p)}$  (\textsection 3), from $B_{2}(k_2)$ to $k.$ This dilogarithm function $\ell i_{2} ^{(p)}$ together with   $\ell i_{2}$ was used to compute cohomology of the Bloch complex in characteristic $p$ in \cite{un-ao}.     The aim of the present work is  to define an infinitesimal Chow dilogarithm on a curve, which reduces to $\ell i_{2} ^{(p)}$ for the projective line. Next we describe the main theorem.

  %was used in \cite{un-ao} to prove that the Bloch complex in characteristic $p$  has the expected cohomology groups. In characteristic $p,$ the dilogarithm function $\ell i_{2} ^{(p)}$ is needed, in addition to the analog of $\ell i_{2},$ since the $K$-group $K_{3}(k_2),$ which is the kernel of the map from $B_{2}(k_2)$ to $\Lambda^2k_2 ^{\times},$ is larger in characteristic $p$ than in characteristic 0 and hence $\ell i_{2}$ does not induce an injective  map from  this group.

Assume that $k$ is a field of characteristic $p\geq 5$ and $C$ is a smooth and proper curve over $k_{2}.$ For each closed point $c$ of $C,$ we fix, once and for all, a closed subscheme $\mathfrak{c}$ of $C$ such that $\mathfrak{c}$ is smooth over $k_{2}$ and has support equal to $c.$ Let $\mathcal{P}$ denote the set of all these subschemes $\mathfrak{c},$ and let $\eta$ be the generic point of $C.$ For a point $x$ of $C$ let $\pazocal{O}_{C,x}$ denote the local ring of $C$ at $x.$ Let $k(C,\mathcal{P})^{\times}$ denote the subgroup of $\pazocal{O}_{C,\eta} ^{\times},$ which consist of  those $f$ in   $\pazocal{O}_{C,\eta} ^{\times}$ such that if  $c$ is a  closed point of $C$ and $\tilde{s} \in \pazocal{O}_{C,c}$ defines the closed subscheme $\mathfrak{c}$ then $f=u\tilde{s}^{n},$ for some $n \in  \mathbb{Z}$ and $u \in \pazocal{O}_{C,c} ^{\times}.$  

In \textsection \ref{section-mod-in-char},  we modify the construction in characteristic 0 in \cite{un-inf} to obtain the infinitesimal Chow-dilogarithm
$$
\rho: \Lambda^{3}k(C,\mathcal{P})^{\times} \to k.
$$
 However, the main construction of this paper is the map 
$$
\rho_{K}: \Lambda ^{3}k(C,\mathcal{P})^{\times} \to k
$$
which is based on the Kontsevich logarithm and is a purely characteristic $p$ object. This is stated as the main theorem: 

\begin{theorem}\label{theorem-main1}
Let $k$ be a field of characteristic $p \geq 5,$ and $C$ a smooth and projective curve over $k_2$ and $\mathcal{P}$ and $k(C,\mathcal{P})^{\times}$  as above. We have a regulator map: 
$$
\rho_{K}: \Lambda ^{3}k(C,\mathcal{P})^{\times} \to k,
$$
associated to the Kontsevich $1\frac{1}{2}$-logarithm $\pounds _1$, which canonically associates an element in $k$ to a triple of functions on $C.$  When  $C$ is the projective line $\mathbb{P}^{1} _{k_2}$ with coordinate function $z,$ then 
$$
\rho_{K}((z-\alpha)\wedge (z-\beta) \wedge (z-\gamma))=a^p\cdot\pounds _1(s),
$$
where $\frac{\gamma-\beta}{\alpha -\beta}=s+as(1-s)t,$
with $a \in k,$ $s \in k^{\times} \setminus \{ 1 \}.$ 
\end{theorem}
\noindent 
We  call the function $\rho_K,$ the Chow-Kontsevich dilogarithm. The   regulators $\rho_{K}$ and $\rho$ are linearly independent. If $t$ denotes the variable in $k_{2},$   rescaling $t$ to $a\cdot t,$ results in multiplying $\rho$ with $a^{3}$ whereas $\rho_K$ is multiplied with $a^p.$ 

One can interpret the existence of  $\rho_K$ as part of a strong reciprocity law on curves. Let $X$ be a smooth and projective curve over a  field $k. $ Suslin proved that the sum of norms of the residue maps    
\[
\begin{CD}
K_{n} ^{M}(k(X)) @>{\oplus_{x \in |X|} \rm res}_x>> \oplus _{x \in |X|} K_{n-1} ^{M}(k(x)) @>{\oplus _{x \in |X|}N_{k(x)/k}}>> K_{n-1} ^M(k)
\end{CD}
\]
from the Milnor $K$-group of the function field of $X$ to those of the closed points of $X,$ is equal to 0 \cite{sus}. When $n=1,$ this states that the degree of the divisor of a rational function is 0. When $n=2,$ this is a restatement of Weil reciprocity. We are interested in the case when $n=3.$ Assuming that $k$ is algebraically closed, Goncharov conjectures in  \cite{arak} a stronger version of this reciprocity law: the sum of the residue maps from the motivic complex of $k(X)$  to that of  $k$ is homotopic to 0.  In particular, this  implies the existence of a map from $\Lambda ^{3}k(X)^{\times}$ to the Bloch group $B_{2}(k)$ (\textsection 2.1). Composing with any map from $B_{2}(k)$ to a group $A$ gives a map from $\Lambda ^{3}k(X)^{\times}$ to $A.$ 

The analog of Goncharov's conjecture in our main set-up  would give us a map $\Lambda ^{3} k(C,\mathcal{P})^{\times} \to B_{2}(k_2)$ (cf. \cite[\textsection 3.4]{un-inf}).   
Composing with the characteristic $p$ dilogarithm   $\ell i _{2} ^{(p)}:B_{2}(k_{2})\to k,$ would give a map from $\Lambda ^{3} k(C,\mathcal{P})^{\times}$ to $k.$ The Chow-Kontsevich dilogarithm  $\rho_{K}$ is expected to be this function.

We will also  use $\rho_K$ to define an infinitesimal invariant of cycles. Let $k_{\infty} :=k[[t]],$ and 
$\square ^n _{k_{\infty}}:=k_{\infty}\times _k (\mathbb{P}^{1} _{k} \setminus \{ 1\})^n.$ For $0\leq q,\, n,$ we define a certain subgroup $\underline{z}^{q} _{f} (k_{\infty},n)$ of  cycles of codimension $q$  in $\square ^n _{k_{\infty}}$ (\textsection 7.2.1), \cite[\textsection 4]{un-inf}.  For fixed $q,$ these groups form a complex. Based on the construction of the map in Theorem \ref{theorem-main1}, we define a regulator map $$
\rho_{K}:\underline{z}^{2} _{f} (k_{\infty},3) \to k,
$$
which vanishes on boundaries and has the property that if two cycles are congruent modulo $(t^2)$ then they have the same image under $\rho_{K}$ (Theorem \ref{modulus theorem}).

{\bf Outline.} In \textsection 2.1 and \textsection 2.2 , we  review the definitions of the Bloch group and the  Bloch complex, the additive dilogarithm of \cite{un-add} and the construction of the infinitesimal Chow dilogarithm  in the characteristic 0 case \cite{un-inf}. In \textsection 2.3, we describe the modifications that are needed in order to carry this construction to characteristic $p.$ In \textsection 3, we review the construction of $\ell i_{2} ^{(p)},$ the additive dilogarithm in characteristic $p,$ based on the Kontsevich $1\frac{1}{2}$-logarithm    and its relation to $K$-theory \cite{un-ao}. In \textsection 4, we describe the construction of the 1-form $\Omega^{(p)}$ which helps us control different local liftings of a curve. This section handles the split case. In \textsection 5, we prove that the residues of the 1-form $\Omega^{(p)}$ are independent of the parametrization: the main result is stated as Proposition \ref{invariance-prop}, the main step of the proof is Lemma \ref{exact-lemma}. In \textsection 5.5, we use this invariance of residues to define the residue of the 1-form $\Omega^{(p)}$ for a pair of  smooth algebras of dimension 1 over $k_p,$ which have isomorphic local rings at their generic points. Lemma \ref{lemma-indep-reduced} shows that this residue depends only on the reduction of this isomorphism to the reduced closed subscheme. In \textsection 6, we relate the additive dilogarithm $\ell i_{2} ^{(p)}$ to the residue of $\Omega^{(p)}$ in Proposition \ref{prop-res-add}. In \textsection 7.1, we prove Theorem \ref{theorem-main1}, which constructs the Chow-Kontsevich dilogarithm and use this construction in \textsection 7.2 to  define an infinitesimal invariant of   cycles in $\underline{z}^{2} _{f} (k_{\infty},3).$

{\bf Notation.} Let $A$ be a ring and $I$ be an ideal of $A.$ If $a \in A,$ we let $a|_{I} \in A/I$ denote the reduction of $a$ modulo $I.$ If $p=(a,b,c) \in A^{ 3}$ is a triple of elements in $A,$  we write $p|_{I}$ for $(a|_{I}, b|_{I},c|_{I}) \in (A/I)^{ 3}.$ If the ideal is $I=(t^m), $ we also write $a|_{t^m}$ instead of $a|_{(t^m)}$.   If $\alpha :X\to Y$ is any function and $p=(a,b,c) \in X\times X \times X,$ we abuse the notation and write $\alpha (p)$ for $(\alpha(a),\alpha(b),\alpha(c)).$ 

For a ring $R,$ we let $R_{\infty}:=R[[t]]$ be the formal power series ring over $R$  and $R_{m}:=R[t]/(t^{m}),$ the truncated polynomial ring over $R$ of modulus $m.$ If $R$ is a $\mathbb{Q}$-algebra then the exponential map is defined as  $e^{\alpha}:=\sum _{0\leq n} \frac{\alpha ^n}{n!}$ for $\alpha \in (t)\subseteq R_{\infty}.$ The same formula defines a map for $\alpha \in (t)\subseteq R_{m}.$ On the other hand, if $R$ is a ring of characteristic $p,$ even though there is no such an exponential map on $R_{\infty},$ we define the modified version $\underline{e}^{\alpha}:=\sum _{0\leq n<p}\frac{\alpha^n}{n!},$ for $\alpha \in (t) \subseteq R_{\infty}.$ This gives a map from $tR_{\infty}$ to $R_{\infty}.$ We denote the induced map from  $tR_{p}$ to $R_{p}$ with the same symbol.

\section{The infinitesimal Chow dilogarithm}

\subsection{Definitions of the complexes}

If $R$ is any ring,  we let $R^{\flat}:=\{r \in R|r(1-r) \in R^{\times} \}.$    The Bloch group $B_{2}(R)$ is defined to be the quotient of $\mathbb{Z}[R^{\flat}]$ by  the subgroup generated by 
\begin{align}\label{5term}
 [x]-[y]+[y/x]-[(1-x^{-1})/(1-y^{-1})]+[(1-x)/(1-y)],    
\end{align}
 for all $x,y \in R^{\flat} $ with $x-y \in R^{\times}.$
Let 
$
\delta:B_{2}(R) \to \Lambda^2 R^{\times} 
$ 
denote the map which is defined on the generators by    $\delta([x]):=(1-x)\wedge x.$ The  complex obtained by putting $B_{2}(R)$ in degree 1 and $\Lambda^2 R^{\times}$ in degree 2 is called {\it the Bloch complex of weight 2}:
\[
\begin{CD}
 B_{2}(R) @>{\delta}>> \Lambda^2 R^{\times}.
\end{CD}
\]

\subsection{The  characteristic 0 case} 
In this section, we review the theory in characteristic 0. We refer the reader to \cite{un-add} for  the details on the additive dilogarithm and to \cite{un-inf} and \cite{un-higher}, for the details on the infinitesimal Chow dilogarithm. 

For a  ring $R,$ let 
 $R_{\infty}:=R[[t]]$ denote the ring of  formal power series over $R.$   If  $R$ is a $\mathbb{Q}$-algebra, let   $\log: (1+tR_{\infty})^{\times } \to R_{\infty}$ denote the logarithm  given by
 $\log (1+z):=\sum _{1\leq n} (-1)^{n+1}\frac{z^n}{n},$ for $z \in tR_{\infty}.$ Let $\log ^{\circ}: R_{\infty} ^{\times} \to R_{\infty},$ be the branch  of the logarithm associated to  the  splitting of $R_{\infty} \twoheadrightarrow R$ corresponding to the inclusion $R \hookrightarrow R_{\infty}.$ In other words, $\log ^{\circ}$ is   defined as $\log ^{\circ}(\alpha):=\log (\frac{\alpha}{\alpha(0)}),$ for $\alpha \in R_{\infty} ^{\times}.$ For $q=\sum _{0\leq i} q_i t^i \in R_{\infty},$  we  define   $t_i(q):=q_i.$ For $\alpha \in R_{\infty} ^{\times},$  we let $\ell_i(\alpha):=t_i(\log^{\circ}(\alpha)).$

\subsubsection{Additive dilogarithm } If $R$ is a $\mathbb{Q}$-algebra, and $m\geq 2,$ we defined maps 
$$
\ell i_{m,r}: B_{2}(R_m) \to R,
$$
for each $m<r<2m.$  In this paper, we will only use this map for $m=2$ and $r=3$ and with the notation  $\ell i_{2}$ instead of $\ell i _{2,3}. $ Hence $\ell i_{2}$ defines a map 
$$
\ell i_{2}: B_{2}(R_{2}) \to R.
$$
Explicitly, this map is given by 
\begin{align}\label{char0form1}
\ell i_{2}([s+at])=-\frac{a^3}{2s^2(1-s)^2}.    
\end{align}
This can also be described using the map $\delta: B_{2}(R_{\infty}) \to \Lambda ^{2}R_{\infty} ^{\times}$ as follows. If we let 
$$
\ell_{2}\wedge \ell _1: \Lambda ^{2}R_{\infty} ^{\times}\to R
$$
be given by $(\ell_{2}\wedge \ell _1)(a\wedge b):=\ell_2(a)\ell_1(b)-\ell_2(b)\ell_1(a)$ then 
the map $(\ell_{2}\wedge \ell _1 ) \circ \delta$ factors through the canonical surjection $B_{2}(R_{\infty}) \to B_{2}(R_{2}):$

 \[
\begin{CD}
 B_{2}(R_{\infty}) @>{\delta}>> \Lambda^2 R_{\infty}^{\times}\\
 @VVV   @VV{\ell_{2}\wedge \ell _1}V \\
  B_{2}(R_{2}) @>{\ell i_{2}}>> R,
\end{CD}
\]
 and, abusing the notation, we write 

 \begin{align}\label{char0form2}
 \ell i_{2}= (\ell_{2}\wedge \ell _1 ) \circ \delta.   
 \end{align}
In the following, we will also use the notation  $\ell:=\ell_2 \wedge \ell _1$  and hence write 
 \begin{align}\label{char0form2'}
 \ell i_{2}= \ell \circ \delta.   
 \end{align}

\subsubsection{Infinitesimal Chow dilogarithm}\label{subsubsection-inf-chow}
In this section, we  review the infinitesimal Chow dilogarithm  when the base field $k$ is of characteristic 0. We refer the reader to  \cite{un-inf} and \cite{un-higher} for the details.

\begin{definition} Let $\pazocal{S}$ be a smooth algebra of relative dimension 1 over $k_m,$ with $m \in \mathbb{N} \cup \{ \infty\},$  such that the reduction $\underline{\pazocal{S}}$ of $\pazocal{S}$ modulo $(t)$ is a discrete valuation ring.  Let $c$ be the closed point of ${\rm Spec} (\pazocal{S}).$ We say that a closed subscheme $\mathfrak{c}$ of ${\rm Spec} (\pazocal{S})$ is a {\it smooth lifting} of $c,$ if $\mathfrak{c}$  is smooth over $k_m$ and is supported on  $c.$ If $\tilde{s} $ is an element of $\pazocal{S}$ such that its reduction in $\underline{\pazocal{S}}$ is a uniformizer, we also call $\tilde{s}$ a {\it uniformizer}.  If $\eta$ is the generic point of ${\rm Spec}(\pazocal{S})$ and $\pazocal{S}_{\eta}$ is its local ring  at $\eta,$ then we let 
$$
(\pazocal{S},\tilde{s})^{\times}:=\{\alpha \in \pazocal{S} _{\eta} ^{\times}| \alpha=u\tilde{s}^n,\; {\rm for\; some} \;\; u \in \pazocal{S}^{\times}\; {\rm and} \; n \in \mathbb{Z}  \}.
$$
If the closed subscheme on $\pazocal{S}$ defined by $\tilde{s}$ is $\mathfrak{c},$ then we say that $\tilde{s}$ is a {\it uniformizer for} $\mathfrak{c}.$ In this case, we also write $( \pazocal{S},\mathfrak{c})^{\times}:=(\pazocal{S},\tilde{s})^{\times}.$ \end{definition}
\noindent Note that since any two uniformizers for  $\mathfrak{c}$ differ by multiplication by an element in $\pazocal{S} ^{\times},$ the definition of $(\pazocal{S},\mathfrak{c})^{\times}$ is independent of the choice of the uniformizer. We say that an element $\alpha \in \pazocal{S}_{\eta}$ is {\it good with respect to } $\mathfrak{c}$ or equivalently is $\mathfrak{c}${\it -good} if $\alpha \in (\pazocal{S},\mathfrak{c})^{\times}.$ If $C$ is a smooth curve over $k_m$ and $c$ is a closed point of $C,$ the analogous notions are defined by taking $\pazocal{S}$ as the local ring of $C$ at $c.$

\iffalse
\begin{definition}
Suppose that $\pazocal{S}$ and $\mathfrak{c}$  are as above.  We write 
$$
(\pazocal{S},\mathfrak{c})^{\flat}:=\{f \in (\pazocal{S},\mathfrak{c})^{\times}|1-f \in (\pazocal{S},\mathfrak{c})^{\times}  \}.
$$  We define $B_{2}(\pazocal{S},\mathfrak{c})$ to be the quotient of $\mathbb{Z}[(\pazocal{S},\mathfrak{c})^{\flat}]$   modulo the five term relations (\ref{5term}) associated to pairs $x$ and $y$ in $(\pazocal{S},\mathfrak{c})^{\flat}$ with the property that $x-y\in (\pazocal{S},\mathfrak{c})^{\times}.$
\end{definition}
\fi

Let  $\pazocal{R}$ be any smooth algebra of relative dimension 1 over $k_{m},$ whose reduction $\underline{\pazocal{R}}$ is not necessarily a discrete valuation ring, and $c$ be a closed point of ${\rm Spec} (\pazocal{R}).$ We define a  smooth lifting  of $c$ to ${\rm Spec} (\pazocal{R})$  as a smooth lifting of $c$ to ${\rm Spec} (\pazocal{R}_{c}).$

\begin{definition}
Suppose that  $\pazocal{R}$ is a  smooth algebra  of relative dimension 1 over $k_m,$ with $m \in \mathbb{N}\cup \{ \infty\}.$   
Fix  a smooth lifting $\mathfrak{c}$  of every closed point $c$ of ${\rm Spec} (\pazocal{R})$ and denote  the set of these liftings by $\mathcal{P}.$   We let $$(\pazocal{R},\mathcal{P})^{\times}:=\bigcap_{c \in |{\rm Spec}(\pazocal{R})| }(\pazocal{R}_{c },\mathfrak{c})^{\times}$$
and 
$(\pazocal{R},\mathcal{P})^{\flat}:=\{f \in (\pazocal{R},\mathcal{P})^{\times}|1-f \in (\pazocal{R},\mathcal{P})^{\times}  \}.$  We define $B_{2}(\pazocal{R},\mathcal{P})$ to be the quotient of  $\mathbb{Z}[(\pazocal{R},\mathcal{P})^{\flat}]$   by the group generated by the five term relations (\ref{5term}) associated to pairs $x$ and $y$ in $(\pazocal{R},\mathcal{P})^{\flat}$ with the property that $x-y\in (\pazocal{R},\mathcal{P})^{\times}.$  
\end{definition}

Let  $k(c)$ denote the residue field of $c$  and  $k(\mathfrak{c})$ denote the  ring of regular functions on the affine scheme $\mathfrak{c}.$    We have  a map 
$$\delta:B_{2}(\pazocal{R},\mathcal{P}) \to \Lambda ^2(\pazocal{R},\mathcal{P})^{\times} ,$$ 
as above, whose value on the generator $[x]$ is given by $\delta([x])=(1-x)\wedge x.$ This gives us a complex which depends on the set of liftings $\mathcal{P},$ since each term in the complex does. However, we expect  the cohomology of the complex to be independent of $\mathcal{P}.$ The above $\delta$   induces a map  
$$
\delta:B_{2}(\pazocal{R},\mathcal{P})\otimes(\pazocal{R},\mathcal{P})^{\times}\to \Lambda ^3(\pazocal{R},\mathcal{P})^{\times},$$
which sends $[x] \otimes y $ to $\delta([x])\wedge y.$ Abusing the notation, we denote this map also with the symbol $\delta. $ For each $\mathfrak{c},$ we have residue maps 
\begin{align}\label{res-lam}
    res_{\mathfrak{c}}: \Lambda ^3(\pazocal{R},\mathcal{P})^{\times} \to \Lambda ^{2}k(\mathfrak{c})^{\times},
\end{align}

\begin{align}\label{res-bloch}
res_{\mathfrak{c}}:B_{2}(\pazocal{R},\mathcal{P}) \otimes (\pazocal{R},\mathcal{P})^{\times}\to B_{2}(k(\mathfrak{c})), 
\end{align}
which give a commutative diagram:
$$
\xymatrix{
B_{2}(\pazocal{R},\mathcal{P}) \otimes (\pazocal{R},\mathcal{P})^{\times}\ar^{res_{\mathfrak{c}}}[d] \ar^{\;\;\;\;\;\;\;\delta}[r] & \Lambda ^3(\pazocal{R},\mathcal{P})^{\times} \ar^{res_{\mathfrak{c}}}[d] \\
B_{2}(k(\mathfrak{c})) \ar^{\delta}[r] & \Lambda ^{2}k(\mathfrak{c})^{\times}.
 }
$$
Let us recall these residue maps, which were defined in the classical case in \cite[\textsection 1.14]{geom}. 

The map (\ref{res-lam}) was defined in \cite[\textsection 7]{un-higher} by the following property:  if $\tilde{s}$ denote a uniformizer for $\mathfrak{c}$ and $\alpha _{i} \in  \pazocal{R}_{c }^{\times},$ for $1\leq i \leq 3,$ and $n \in \mathbb{Z}$ then   $res_{\mathfrak{c}}$ maps the term   $(\tilde{s}^n \alpha _{1}) \wedge \alpha _2 \wedge \alpha _{3}$ to $n \cdot \underline{\alpha}_{2} \wedge \underline{\alpha}_{3}.$ Here, if $\alpha$ is in $\pazocal{R}_{c } ^{\times},$     $\underline{\alpha}$ denotes its image  in  $(\pazocal{R}_{c }/(\tilde{s}))^{\times}=k(\mathfrak{c})^{\times}.$ The residue map is well-defined and is independent of the choice of the uniformizer $\tilde{s}.$ 

The map (\ref{res-bloch}) was defined in \cite[\textsection 8.1]{un-higher} (cf. \cite[\textsection 3.3.1]{un-inf}) by the following properties: $res_{\mathfrak{c}}$ sends an element of the form $[\alpha] \otimes \beta$ to 0 if $\alpha \not\in \pazocal{R}_{c } ^{\flat}$ or $\beta \in \pazocal{R}_{c } ^{\times};$ and sends $[\alpha] \otimes \tilde{s}$ to $[\underline{\alpha}],$ if  $\alpha \in \pazocal{R}_{c } ^{\flat}.$ Again the map is independent of the choice of the uniformizer and is well-defined, cf. \cite[\textsection 1.14]{geom}.

Suppose that $C$ is a smooth and projective curve over $k_m$ and  $|C|,$  the set of its closed points and let $\eta$ be its generic point. For each $c \in |C|,$ fix a smooth lifting $\mathfrak{c}$ of $c$ to $C$  and let $\mathcal{P}$ denote the set of all of these liftings. For an open subset $U$ of $C,$ let $\mathcal{P}|_{U}$ denote the set of those $\mathfrak{c} \in \mathcal{P}$ such that $\mathfrak{c} \subseteq U.$ 

Let  $(\pazocal{O}_{C}, \underline{\mathcal{P}})^{\times}$ (resp. $(\pazocal{O}_{C}, \underline{\mathcal{P}})^{\flat}$) be the sheaf on $C$ such that 
$$
(\pazocal{O}_{C}, \underline{\mathcal{P}})^{\times}(U)=(\pazocal{O}_{C}(U), \mathcal{P}|_{U})^{\times}\;\; {\rm  (resp.}\;\;
(\pazocal{O}_{C}, \underline{\mathcal{P}})^{\flat}(U)=(\pazocal{O}_{C}(U), \mathcal{P}|_{U})^{\flat}
) $$
for each affine open subset $U$ of $C.$ Similarly, let $B_{2}(\pazocal{O}_{C},\underline{\mathcal{P}})$ be the sheaf associated to the presheaf whose sections on such $U$ are $B_{2}(\pazocal{O}_{C}(U),\mathcal{P}|_{U}).$

This gives us a complex $\mathcal{C}(C,\mathcal{P})$ of sheaves on $C$ which are concentrated in degrees 2 and 3: 
\begin{align}\label{main-complexofsheaves-inf}
 B_{2}(\pazocal{O}_C, \underline{\mathcal{P}}) \otimes (\pazocal{O}_{C},\underline{\mathcal{P}}) ^{\times}\to \oplus _{c \in |C|}i_{c*}(B_{2}(k(\mathfrak{c}))) \oplus \Lambda ^3 (\pazocal{O}_{C},\underline{\mathcal{P}})^{\times},   \end{align}
where $i_{c}:c \to C$ is the inclusion map. Let  $k(C)^{\times}:=\pazocal{O}_{C,\eta} ^{\times}$ denote group of units of the local ring of $C$ at its generic point $\eta$ and  $k(C,\mathcal{P})^{\times} := \Gamma(C,(\pazocal{O}_{C},\underline{\mathcal{P}})^{\times}) \subseteq k(C)^{\times}$ denote those which are also $\mathcal{P}$-good. 

When $m=2,$ in other words when $C/k_{2}$ is a smooth and projective curve, the infinitesimal Chow dilogarithm $\rho$ is a map
$
\rho: \Lambda^{3} k(C,\mathcal{P} )^ \times \to k.
 $

Suppose that  $\pazocal{A}/k_{\infty}$ is a smooth  algebra over $k_{\infty}$ of relative  dimension 1, and $\mathcal{P}$ is a set of smooth liftings of closed points of ${\rm Spec} (\pazocal{A})$ as above. For  $\mathfrak{c} \in \mathcal{P},$ recall that $k(\mathfrak{c})$ denotes the ring of functions on the affine scheme $\mathfrak{c}.$ If $k'$ denotes the residue field of $\mathfrak{c},$ there is a unique isomorphism of $k_{\infty}$-algebras between $k(\mathfrak{c})$ and $k'_{\infty}$ which is the identity map modulo $(t).$ Therefore we can identify $k(\mathfrak{c})$ and $k'_{\infty}.$ In particular, this gives us a canonical map $\ell: \Lambda ^{2}k(\mathfrak{c})^{\times} \to k',$ which corresponds to $\ell =\ell_{2} \wedge \ell _1 : \Lambda ^{2}(k'_{\infty} )^{\times} \to k'$ via the above identification $k(\mathfrak{c})=k'_{\infty}.$

If $\tilde{f}, $ $\tilde{g},$  $\tilde{h} \in (\pazocal{A},\mathcal{P})^{\times}$ and $\mathfrak{c} \in \mathcal{P}$ then 
$
res_{\mathfrak{c}} (\tilde{f} \wedge \tilde{g} \wedge \tilde{h}) \in \Lambda ^2 k(\mathfrak{c})^{\times}.
 $
Applying $\ell$ to this object we obtain an element in $k'$ and taking its trace  to $k,$ we obtain 
$$
{\rm Tr}_k(\ell (res_{\mathfrak{c}} (\tilde{f} \wedge \tilde{g} \wedge \tilde{h}))) \in k.
$$
This term will be essential in defining the local contribution to the Chow dilogarithm. 

Suppose that    $C/k_{2}$ has a global lifting $\tilde{C}/k_{\infty}.$ Namely, $\tilde{C}$ is   a smooth and projective curve over $k_{\infty}$ together with a fixed isomorphism between $\tilde{C} \times_{k_{\infty}}k_{2}$ and $C.$ Let $\tilde{\mathcal{P}}$ be a set of smooth liftings for each point of $c \in |C|$ to $\tilde{C}$ which reduce to the liftings in $\mathcal{P}$ modulo $(t^2).$ Suppose further that $\tilde{f},$ $\tilde{g}$,  $\tilde{h} \in k(\tilde{C},\tilde{\mathcal{P}})^{\times}$ reduce to $f,$ $g$ and $h$ modulo $(t^2).$ 
Then  
\begin{eqnarray}\label{exp1-rho}
\rho(f\wedge g\wedge h)=\sum _{c \in |C|} {\rm Tr}_k(\ell(res_{\tilde{c}} (\tilde{f} \wedge \tilde{g} \wedge \tilde{h}))),
\end{eqnarray}
where $\tilde{c} \in \tilde{\mathcal{P}}$ denotes the lifting of $c.$ 
 
In general,  such global liftings  do not exist. Even when they exist, we do not know, a priori,  that (\ref{exp1-rho}) is independent of the choice of the liftings. 
Our method to define $\rho$ in general, is to choose local   liftings of the curve  and  the  functions and also to choose a generic lifting of the curve and the functions 
and then to measure the defects between the local liftings and the generic lifting. This defect between the different liftings will be measured by the residue of a 1-form. Next we  describe this in detail.

We start with the following data. Let $\tilde{\pazocal{R}}$ and  $\hat{\pazocal{R}}$ are smooth $k_{\infty}$-algebras  of relative dimension one, and $\chi$ an isomorphism $\chi: \tilde{\pazocal{R}}/(t^2) \xrightarrow{\sim} \hat{\pazocal{R}}/(t^2)$ of $k_{2}$-algebras. Suppose further that we a have a triple $\tilde{p}:=(\tilde{f},\tilde{g},\tilde{h})$ of functions in $\tilde{\pazocal{R}} ^{\times}$ and a triple $\hat{p}:=(\hat{f},\hat{g},\hat{h})$ of functions in $\hat{\pazocal{R}} ^{\times},$ such that $\chi (\tilde{p}|_{t^2})=\hat{p}|_{t^2}.$    To this data, we will attach an element 
 $\omega(\tilde{p},\hat{p},\chi) \in \Omega^{1}_{\underline{\hat{\pazocal{R}}}/k},$ where $\underline{\hat{\pazocal{R}}}$ is 
 the reduction modulo $(t)$ of $\hat{\pazocal{R}}.$ 
 
 We will proceed as follows. 
  Let $\overline{\chi}:\tilde{\pazocal{R}} \xrightarrow{\sim} \hat{\pazocal{R}}$ be any $k_{\infty}$-algebra isomorphism which reduces to $\chi$ modulo $(t^2)$  and $\varphi: \underline{\hat{\pazocal{R}}} \to \hat{\pazocal{R}}$ be any splitting of the canonical projection.   Denote by $\overline{\varphi}$ the corresponding isomorphism $\underline{\hat{\pazocal{R}}}[[t]] \xrightarrow{\sim} \hat{\pazocal{R}}$ of $k_{\infty}$-algebras,  which extend $\varphi.$ 
Then we define: 
$$
\omega(\tilde{p},\hat{p},\chi):= \Omega( \overline{\varphi}^{-1}( \overline{\chi}(\tilde{p})), \overline{\varphi}^{-1}(\hat{p}) ),
$$
with $\Omega$ as below. 

Let $\tilde{q}=(\tilde{y}_{1},\tilde{y}_{2},\tilde{y_{3}})$ and $\hat{q}=(\hat{y}_{1},\hat{y}_{2},\hat{y_{3}}),$ with $\tilde{y}_i, \, \hat{y}_i \in \underline{\hat{\pazocal{R}}}[[t]] ^{\times},$ and $\hat{y}_{i}-\tilde{y}_i \in (t^2),$ for all $1 \leq i \leq 3.$ There are unique $\alpha_{0i} \in \underline{\hat{\pazocal{R}}}\, ^{\times}$ and $\alpha _{1i}, \hat{\alpha}_{ji},\tilde{\alpha}_{ji} \in \underline{\hat{\pazocal{R}}} ,$ for $1\leq i \leq 3$ and $2\leq j$ such that 
$$
\hat{y}_i=\alpha_{0i}e^{t\alpha_{1i}+t^2\hat{\alpha}_{2i}+\cdots }
$$
and
$$
\tilde{y}_i=\alpha_{0i}e^{t\alpha_{1i}+t^2\tilde{\alpha}_{2i}+\cdots }.
$$
We then define 
\begin{eqnarray*}
\Omega (\tilde{q},\hat{q}) := \sum _{\sigma \in S_{3}} (-1) ^\sigma \alpha_{1 \sigma(1)} ( \tilde{\alpha}_{2\sigma(3)}-\hat{\alpha}_{2\sigma(3)} ) \cdot d \log (\alpha_{0\sigma(2)}) \in \Omega^{1} _{\underline{\hat{\pazocal{R}}}/k}.
\end{eqnarray*}
It turns out that the definition of $\omega(\tilde{p},\hat{p},\chi)$ is   independent of the choices of the lifting $\overline{\chi}$ and the splitting $\varphi.$ It does depend on the liftings of the triples of functions and on $\chi$, as  reflected in the notation.

Suppose that $C/k_{2}$ and $\mathcal{P}$ are as above and $f, g, h$ are in $k(C,\mathcal{P})^{\times}.$ We will describe the definition of $\rho (f \wedge g \wedge h) \in k$ below. 

First let $p:=(f,g,h)$ and let  $\eta$ be the generic point on $C.$ Let  $\tilde{\pazocal{A}}$ be a smooth $k_{\infty}$-algebra together with  an isomorphism 
$
\alpha: \tilde{\pazocal{A}}/(t^2)\xrightarrow{\sim} \pazocal{O}_{C,\eta}.
 $
 Let $\tilde{p}_{\eta}$ be a triple of functions in $\tilde{\pazocal{A}},$ whose reductions modulo $(t^2)$ map to the germs $p_{\eta}$ of the functions $p$ at $\eta.$ For each $c\in |C| ,$ let $\widetilde{\pazocal{B}}^{\circ}_{c}$ be  a smooth $k_{\infty}$-algebra together with an isomorphism $
\tilde{\gamma}_{c}: \widetilde{\pazocal{B}}^{\circ}_{c}/(t^2) \xrightarrow{\sim}  \hat{\pazocal{O}}_{C,c},
$ 
from the reduction of $\widetilde{\pazocal{B}}^{\circ}_{c}$ modulo $(t^2)$ to the completion of the local ring of $C$ at $c.$ Let $\mathcal{P}_{c}=\{ \mathfrak{c}\}$ denote the smooth lifting of the point $c$ to $C.$ Let $\mathcal{Q}_{c}:=\tilde{\gamma}_{c} ^{*}(\mathcal{P}_{c})$ denote the corresponding smooth lifting on ${\rm Spec}(\widetilde{\pazocal{B}}^{\circ}_{c})$ obtained by transport of structure via the isomorphism $\tilde{\gamma}_{c}.$ If $p_{c}$ denotes the triple of functions obtained by taking the germs of the functions in $p$ at the point $c$ then $p_c \in ((\hat{\pazocal{O}}_{C,c},\mathcal{P}_{c})^{\times} )^{3}.$ With the notation above, we have $q_c:=\tilde{\gamma}_{c} ^{-1}(p_{c}) \in ((\widetilde{\pazocal{B}}^{\circ}_{c}/(t^2),\mathcal{Q}_{c})^{\times })^3.$ Let $\tilde{\mathcal{Q}_{c}}$ be a smooth lifting  to ${\rm Spec}(\widetilde{\pazocal{B}}^{\circ}_{c})$ which reduces to $\mathcal{Q}_{c}$ modulo $(t^2).$ 
Let $\tilde{q}_{c} \in ((\widetilde{\pazocal{B}}^{\circ}_{c},\tilde{\mathcal{Q}}_{c})^{\times })^3
$ which reduces to $q_{c}$ modulo $(t^{2}).$ 
If $\tilde{q}_{c}=(\alpha, \beta, \gamma),$ and $\tilde{\mathcal{Q}}_{c}=\{ \tilde{\mathfrak{c}}\}$ then $res_{\tilde{\mathfrak{c}}}(\alpha \wedge \beta \wedge \gamma) \in \Lambda ^{2}k(\tilde{\mathfrak{c}})^{\times}.$ We denote the last expression $res_{\tilde{\mathfrak{c}}}(\alpha \wedge \beta \wedge \gamma)$ by $res_{c}(\tilde{q}_c)$ in order to simplify the notation.

If we let $\tilde{\gamma}_{c,\eta} ^{-1} \circ \alpha_c$ denote the isomorphism
obtained by composing the  completion of  $\alpha$ at $c$ and the localization of the inverse of  $\tilde{\gamma}_{c}$ at the generic point then the value of $\rho$ on $p$ is given by:

\begin{eqnarray}\label{ord-chow-dilog-eq}
\rho(p):= \sum_{c\in |C|}{\rm Tr}_k(\ell(res_c(\tilde{q}_{c}) )+ res_{c} \omega(\tilde{p}_{\eta},\tilde{q}_{c} ,  \tilde{\gamma}_{c,\eta} ^{-1} \circ \alpha_c )).
\end{eqnarray}
Here the term $ \omega(\tilde{p}_{\eta},\tilde{q}_{c} ,  \tilde{\gamma}_{c,\eta} ^{-1} \circ \alpha_c )$ should be interpreted in the following sense. The map $ \tilde{\gamma}_{c,\eta} ^{-1} \circ \alpha_c $ is an isomorphism form  the completion $\hat{\tilde{\pazocal{A}}}_{c}/(t^2)$ of $\tilde{\pazocal{A}}/(t^2)$ at $c$ to the localization $\widetilde{\pazocal{B}}^{\circ}_{c,\eta}/(t^2)$
 of $\widetilde{\pazocal{B}}^{\circ}_{c}/(t^2)$ at the generic point $\eta. $ The images of the triples of functions $\tilde{p} _{\eta}$ and $\tilde{q}_{c}$ in the completions and localisations give liftings to $\hat{\tilde{A}}_{c}$ and to $\widetilde{\pazocal{B}}^{\circ}_{c,\eta}.$ This is precisely the set-up in which we can use $\omega$ to obtain a 1-form $\omega(\tilde{p}_{\eta},\tilde{q}_{c} ,  \tilde{\gamma}_{c,\eta} ^{-1} \circ \alpha_c )$ on $\widetilde{\pazocal{B}}^{\circ}_{c,\eta}/(t).$ The residue of this 1-form at the closed point of  $\widetilde{\pazocal{B}}^{\circ}_{c}/(t)$ is an element in the residue field of $c.$ 

It is the main theorem of \cite{un-inf} that  the above  sum  is finite and is  independent of all the choices involved. Moreover, $\rho$ induces a map 
from $\Lambda ^3 k(C,\mathcal{P}) ^{\times}$ to $k.$

\subsection{Modifications in the characteristic $p$ case}\label{section-mod-in-char} Let $p$ be a prime number and suppose that $R$ is a ring of characteristic $p,$ i.e. $\mathbb{F}_{p} \subseteq R.$   We cannot define a logarithm map from $(1+rR_{\infty})$ to $R_{\infty}$ since the power series expression of the logarithm map has denominators.  On the other hand, there is a well-defined logarithm homomorphism $$
\log: (1+tR_{m})^{\times } \to R_{m}
$$ 
given by $\log (1+z)=\sum _{1\leq n<p} (-1)^{n+1}\frac{z^n}{n},$ for $z \in tR_{m},$  when  $m\leq p.$ 
 Similarly, the map $\log^{\circ}:R_{m} ^{\times} \to R_{m}$ is defined by $\log ^{\circ}(\alpha)=\log (\frac{\alpha}{\alpha(0)})$ and is a homomorphism. As above, the maps $\ell _i: R_{m} ^{\times } \to R$ are defined 
 by the formula $\ell _{i}(\alpha)=t_{i}(\log ^{\circ}(\alpha)),$ for $1\leq i <m.$

 We can define a map $\ell i_{2}: B_{2}(R_{2}) \to R$ using the formula (\ref{char0form1}), when $p>2.$  More explicitly, the map $\ell :=\ell _{2} \wedge \ell _1: \Lambda^{2}R_{3} ^{\times} \to R$ is used to define $\ell i_{2}: B_{2}(R_{2}) \to R$ by the following commutative diagram 
  \[
\begin{CD}
 B_{2}(R_{3}) @>{\delta}>> \Lambda^2 R_{3}^{\times}\\
 @VVV   @VV{\ell_{2}\wedge \ell _1}V \\
  B_{2}(R_{2}) @>{\ell i_{2}}>> R,
\end{CD}
\]
exactly analogous to (\ref{char0form2}). 

If we assume now that $C/k_{2}$ is a smooth and proper curve, where $k$ is a field of characteristic $p>2$ then the construction in the previous section carries over in this case to give a map $\rho: \Lambda ^{3} k(C,\mathcal{P})^{\times} \to k.$  In the construction, we need to make the following  modifications. In the characteristic 0 case we choose liftings $\tilde{\pazocal{A}}$ and $\tilde{\pazocal{B}}^{\circ}_{c}$ which are smooth over $k_{\infty}.$ In the characteristic $p$ case, we will choose these liftings to be smooth $k_{3}$-algebras. Also in the definition of $\omega,$ we will start with smooth $k_{3}$-algebras $\hat{\pazocal{R}}$ and $\tilde{\pazocal{R}}.$ Then $\overline{\chi}$  will be a morphism of $k_{3}$-algebras and $\overline{\varphi}$ will be an isomorphism from $\underline{\hat{\pazocal{R}}}_{3}$  to $\hat{\pazocal{R}},$ where  $\underline{\hat{\pazocal{R}}}:=\hat{\pazocal{R}}/(t).$ Finally, we note that the definition of $\Omega$ works as in characteristic 0. In the definition of $\Omega,$ in order to obtain the coefficients $\alpha_{1i},$ $\hat{\alpha}_{2i}$ and $\tilde{\alpha}_{2i},$ we are essentially using $\ell _1$ and $\ell_{2},$ which also makes sense on  $\underline{\hat{\pazocal{R}}}_3.$

 \section{The additive dilogarithm in characteristic $p$ }

From now, we assume that $p \geq 5$ throughout the paper. When $R$ is a ring of characteristic $p,$ the additive dilogarithm $\ell i _{2} :B_{2}(R) \to R$ does not tell the whole story.  There is another function, which is of  characteristic $p$ in nature that completes the picture. Such a function was constructed in \cite{un-ao}. We describe this function in this section. 

 The construction is based on the  $1\frac{1}{2}$-logarithm  defined by Kontsevich in \cite{kont}. Let us first recall this function which is defined as: 
 
 $$  
 \pounds_{1}(s)=\sum_{1 \leq i \leq p-1}\frac{s^{i}}{i},
 $$
 for $s \in R.$ If we define
$$
\ell i_{2} ^{(p)}([s+\alpha t]):=\frac{\alpha^p}{s^p(1-s)^p}\pounds _1 (s) =\frac{\alpha^p}{s^p(1-s)^p}\sum_{1 \leq i \leq p-1}\frac{s}{i},
$$
for $s+\alpha t \in R_{2} ^{\flat},$  we deduce that $\ell i_{2} ^{(p)}$ induces a map: 
$$
\ell i_{2} ^{(p)}: B_{2}(R_{2}) \to R
$$
using the functional equations for $\pounds_{1}$ \cite{un-ao}. We  note that with the notation of \cite{un-ao}, we have $\ell i_{2} ^{(p)}=\mathfrak{Li}_{2} ^{p}.$   Because of the context,  it is not possible to  confuse the notation $\ell i_{2} ^{(p)}$  for the one for the  $p$-adic dilogarithm.

Similar to $\ell i_2$ the map $\ell i_{2} ^{(p)}$  can be expressed in terms of the differential in the Bloch complex. Namely, the diagram: 

$$
\xymatrix{
B_{2}(R_p)    \ar^{\delta}[r] \ar@{->>}[d] & \Lambda ^2 R_{p} ^{\times} \ar^{ \frac{1}{2} \sum _{1 \leq i <p }i\cdot \ell _{p-i} \wedge \ell _{i}}[d] \\   B_{2}(R_2) \ar^{\ell i_{2} ^{(p)}}[r] & R ,}
$$
commutes, allowing  us to  write  
\begin{align}\label{pound-delta}
\ell i_{2} ^{(p)}=(\frac{1}{2} \sum _{1 \leq i <p }i\cdot \ell _{p-i} \wedge \ell _{i}) \circ \delta. 
\end{align}
We put $\ell^{(p)}:=\frac{1}{2} \sum _{1 \leq i <p }i\cdot \ell _{p-i} \wedge \ell _{i},$ and rewrite the above expression as
\begin{align}\label{pound-delta'}
\ell i_{2} ^{(p)}= \ell^{(p)} \circ \delta. 
\end{align}
The above expression should be thought of as analogous to the expressions (\ref{char0form2}) and (\ref{char0form2'}) which express $\ell i_2$ in terms of $\delta$ and $\ell$ and used to construct the infinitesimal Chow dilogarithm above. Similarly, the  expression (\ref{pound-delta'}) will be used to define a Chow-Kontsevich dilogarithm.

We mentioned in the beginning of this section that, in characteristic $p,$ the additive dilogarithm does not tell the whole story. More precisely, that one does not get an injective regulator map if one restricts $\ell i_{2}$ to the appropriate part of the $K$-group. 
Below is a justification that together with 
$\ell i_{2} ^{(p)},$ they suffice \cite{un-ao}: 
\begin{theorem} Let $k$ be the algebraic closure of $\mathbb{F}_{p}.$ 
The direct sum $\ell i_2\oplus \ell i_{2} ^{(p)}$  induces an isomorphism 
$$
K_{3} ^{\circ} (k_{2}) \to B_{2} ^{\circ}(k_{2})   \to k \oplus k
$$
when restricted to the infinitesimal part of $K_{3}(k_{2}). $ 
 \end{theorem}
 We expect the analog of the above theorem to hold for any field $k$ of characteristic $p\geq 5.$ In the general case, however, one would get an isomorphism  with indecomposable part of $K_{3}(k_2)^{\circ}.$

\section{The comparison  1-form for the Chow-Kontsevich dilogarithm}

One of the most  essential steps in defining the infinitesimal Chow dilogarithm was the construction of a 1-form that  compares  the different choices of liftings for the parameters. More precisely, this 1-form $\Omega$ has the  property that for two different choices of liftings, the difference of the $\ell$ values of their  residues can be expressed in terms of the residue of $\Omega$ \cite[Proposition 2.4.4]{un-inf}. In characteristic $p,$ we will construct a similar 1-form, which will denote with $\Omega^{(p)},$ that will have this property with $\ell^{(p)}$  replaced with $\ell$ in the last sentence.

\subsection{The definition of $\Omega ^{(p)}$ on $\Lambda ^{3}R_{\infty} ^{\times}$}\label{defn-of-omega-sec}
 Suppose that $R$ is a ring of characteristic $p.$  We  define $\Omega^{(p)}$ on $\Lambda ^{3}R_{\infty}^{\times}$ as the unique map 
$$
\Omega ^{(p)}: \Lambda ^{3}R_{\infty}^{\times} \to \Omega ^{1} _{R}
$$
which satisfies the following properties (i)-(vi). 

$(i)$ $\Omega ^{(p)}(u\wedge v \wedge w)=0,$ if $u \in 1+(t^p)R_{\infty}.$

\noindent This last expression implies that the map $\Omega ^{(p)}$ descends to give a map from  $\Lambda ^3R_p ^{\times}.$  Every element of $R_{p}^{\times}$ is represented by products of terms of the form $f$ with $f \in R^{\times} \subseteq R_{p} ^{\times}$ and terms of the form $\underline{e}^{\alpha t^a} \in R_{p} ^{\times},$ with $\alpha \in R$ and $0<a<p.$   It is sufficient to determine $\Omega ^{(p)}$ on the elements of the form $f$ and $\underline{e}^{\alpha t^a}.$  We proceed to do this. 

$(ii)$ $\Omega^{(p)}(f\wedge g \wedge h)=\Omega ^{(p)}(\underline{e}^{\alpha t^a}\wedge g \wedge h)=0,$ for $f, g, h \in R^{\times}$ and $\alpha \in R, $ and $0<a<p.$ 

$(iii)$ $\Omega^{(p)}(\underline{e}^{\alpha t^a}\wedge \underline{e}^{\beta t^b} \wedge h)=0 ,$ if $a+b \neq p;$ and $\Omega^{(p)}(\underline{e}^{\alpha t^a}\wedge \underline{e}^{\beta t^b} \wedge h)=\alpha b\beta \frac{dh}{h}$ for $ h \in R^{\times}$ and $\alpha, \beta \in R, $ and $p>a>b>0,$ with $a+b=p.$ 

$(iv)$  $\Omega^{(p)}(\underline{e}^{\alpha t^a}\wedge \underline{e}^{\beta t^b} \wedge \underline{e}^{\gamma t^c} )=0 ,$ if $a+b+c \neq p.$

$(v)$  $\Omega^{(p)}(\underline{e}^{\alpha t^a}\wedge \underline{e}^{\beta t^b} \wedge \underline{e}^{\gamma t^c} )=\alpha(b \beta \cdot d\gamma-c \gamma \cdot d\beta) \in \Omega ^{1}_{R},$ if $a+b+c=p$ and $a>b\geq c>0.$ Here, $d:R\to \Omega ^{1}_{R}$ denotes the canonical derivation.

$(vi)$  $\Omega^{(p)}(\underline{e}^{\alpha t^a}\wedge \underline{e}^{\beta t^b} \wedge \underline{e}^{\gamma t^c} )=\gamma(a \alpha \cdot d\beta-b \beta \cdot d\alpha) \in \Omega ^{1}_{R},$ if $a+b+c=p$ and $a=b> c>0.$

We can rephrase the above definition much more concisely by slightly abusing the notation. First note that in the above, the contribution coming from the term $h \in R^{\times}$ is $\frac{dh}{h}.$ If we abuse the notation and write $h=\underline{e}^{\gamma},$ then $\frac{dh}{h}=d\gamma$ and this expression resembles the contributions coming from other terms of the form   $\underline{e}^{\alpha t^a}$ with $0<a.$ 

We can then replace the conditions (ii)-(vi) above with the following equivalent formulation. For $p>a\geq b \geq c\geq 0,$ we define $\Omega^{(p)} (\underline{e}^{\alpha t^a}\wedge \underline{e}^{\beta t^b}\wedge \underline{e}^{\gamma t^c})$ as: 

$
(ii)' \;\;\; 0, \;\;\; {\rm if} \;\; a+b+c \neq p
$

$
(iii)'\;\;\;\alpha(b \beta \cdot d\gamma-c \gamma \cdot d\beta), \;\;\; {\rm if} \;\; a+b+c = p, \; {\rm and } 
\; a >b\geq c\geq 0
$

$
(iv)' \;\;\;\gamma(a\alpha \cdot d\beta-b\beta\cdot d\alpha),  \;\;\; {\rm if} \;\; a+b+c = p, \; {\rm and } 
\; a =b> c\geq 0.
$

Note that in the above expression, if $c=0$ even though the expression $\gamma$ does not make sense we set $c\gamma:=0.$  We emphasize that the above definition covers all the possibilities since $p\geq 5.$

\subsection{The definition of $\Omega ^{(p)}$ for  a split algebra}\label{split-def}

If $A$ is a ring  and  $I$ is an ideal of $A,$ we  let $(A,I)^{\times}:=\{(a,b)| a, \, b \in A ^{\times},\; a-b \in I \},$ and let $\pi_i:(A,I)^{\times} \to A^{\times},$ for $i=1,2$  denote the two projections.  We define a map 
$$
\tilde{\Omega} ^{(p)} : \Lambda ^{3}(R_p,(t))^{\times} \to \Omega ^{1} _{R}.
$$
Let $s:\Lambda ^{3}(R_{p},(t))^{\times} \to \Lambda ^{3}R_{p} ^{\times},$ denote  the map given by the difference $s:=\Lambda ^3\pi_{1} -\Lambda ^{3}\pi_2.$ Then we define $\tilde{\Omega} ^{(p)}:=\Omega^{(p)}\circ s$ as the composition of $s$ and $\Omega ^{(p)}.$

\section{The invariance of the  residue of $\Omega^{(p)}$ with respect to reparametrization} In order to generalize the definition of $\Omega ^{(p)}$ to certain rings without a prescribed set of coordinate functions, we need a certain invariance property for $\Omega ^{(p)}.$ 
The aim of this section is to describe and prove this invariance property.  

\subsection{An elementary formula for an infinitesimal automorphism of $k((s))_{\infty}$} For a ring $R,$ let $R((s))$ denote the ring of formal Laurent series in $R,$ i.e. the localization of $R[[s]]$ with respect to the set of non-negative powers of $s.$ Note that we can describe $R((s))_{\infty}=R((s))[[t]]$ as the set of formal series 
$$
\sum _{0\leq j \atop N_{j} \leq i} f_{ij} s^it^j, 
$$
with $N_{j} \in \mathbb{Z},$ for every $0\leq j;$ and $f_{ij} \in R,$ for all $0\leq j$ and $N_j \leq i.$  
We endow $R((s))_{\infty}$ with the topology such that a sequence converges, if for each monomial $s^it^j,$ its coefficient in the sequence stabilizes; and for each $j,$ there is an $N_j \in \mathbb{Z}$ such that the coefficients of $s^it^j$ in each term of the sequence is equal to 0, for every $i<N_j.$ By this description, we see that any continuous automorphism $\sigma$ of the $R_{\infty}$-algebra $R((s))_{\infty}$ satisfies 
$$
\sigma(\sum _{0\leq j \atop N_{j} \leq i} f_{ij} s^it^j)=\sum _{0\leq j \atop N_{j} \leq i} f_{ij} \sigma(s)^i t^j,
$$
and therefore is determined by its value $\sigma(s)$ on $s.$ Given any $x \in R((s))$ and $w \geq 1,$ the element $s+xt^{w}$ is invertible in $R((s))_{\infty},$ since 
$$
\frac{1}{s}\sum _{0\leq n} (-\frac{xt^w}{s})^{n} \in R((s))_{\infty}.
 $$
 Moreover, the sequence $((s+xt^w)^{n})$ converges to 0. These two facts imply that the $\sigma$ defined by the above formula with $\sigma (s)=s+xt^w$ is a continuous automorphism of the $R_{\infty}$-algebra $R((s))_{\infty}.$ 
 
  Returning to our standard set-up, where $k$ is a field of characteristic $p,$ suppose that $\sigma$  is a continuous  $k_{\infty}$-automorphism of $k((s))_{\infty}$  as above such that  $\sigma(s)=s+x t^w$ with $w \geq 1$ and $x\in k((s)).$ We will do an elementary  computation to express the images of $\sigma(f)$ and $\sigma(e^{\alpha t^a})$ in $k((s))_{p},$  where $f \in k((s))^{\times}$ and $\alpha \in k((s))$ and $a\geq 1,$ in terms of the logarithmic derivatives of $f$ and the derivatives of $\alpha.$  
 
The formulas in question do not have $p$-torsion when they are reduced modulo $(t^p)$ as above. In order to prove the formulas, it is easier to prove them first when 
 $k$ is replaced with a ring $R$ which does not have $\mathbb{Z}$-torsion and then deduce the result for $k$ by using an appropriate map $R \to k.$ The advantage of using $R$ is that here one can use the Taylor expansion formula, which has denominators, after passing to $R_{\mathbb{Q}}.$ We start with a lemma which is a restatement of the Taylor expansion formula. 

\begin{lemma}
Suppose that  $R$ is a ring without $\mathbb{Z}$-torsion  and $\sigma$ is the unique continuous  $R_{\infty}$-automorphism of $R((s))_{\infty}$ such that $\sigma (s)=s+xt^w,$ with $x \in R((s))$ and $w\geq 1.$ Then  for $f\in R((s))$ we have 
 $$
 \sigma (f)=\sum _{0\leq i}\frac{(xt^w)^{i}f^{(i)}}{i!} \in R((s))[[t]]\subseteq  R_{\mathbb{Q}}((s))[[t]].
 $$ Here,  $f'$ is defined by $df=f'ds$ and $f^{(i)}:=(f^{(i-1)})'$ with $f^{(0)}:=f,$ for $i\geq 1.$  

\end{lemma}

\begin{proof}
By using $R$-linearity  and continuity, we reduce to the case when $f(s)=s^{n}$ for some $n \in \mathbb{Z}.$ 
By using the map that sends $t$ to $xt^w,$ we reduce to the case where $\sigma(s)=s+t.$ By the naturality of the formula, it suffices to check it  for $R=\mathbb{R}.$     Let $f(s)=s^{n}$ for some $n \in \mathbb{Z}.$ If we first fix $s$ to be a value near 1 and vary $t$ near 0 and expand $f(s+t)$ near $s,$  the  Taylor expansion formula gives  $f(s+t)=\sum _{0\leq i  }\frac{f^{(i)}(s)t^i}{i!}.$ Now letting $s$ vary, we realize that both sides are analytic functions of $s$ and $t$ near 1 and 0 and we have the formula we are looking for between two analytic functions. The result then follows by identifying the analytic functions with their power series expansions. 
 \end{proof}

\begin{lemma}\label{lemma-repar-taylor1} Suppose that $R$ and $\sigma$  are as above. 
Then, for $1\leq a$ and $\alpha \in R((s)),$
$$\sigma(e^{\alpha t^a})=e^{\sum _{0\leq i}\frac{x^i
\alpha^{(i)}}{i!}t^{a+iw}} \in R_{\mathbb{Q}}((s))_{\infty};$$    and, for $f \in R((s)))^{\times},$ 
$$
\sigma(f)=fe^{\sum _{1\leq i}\frac{x^i(f'/f)^{(i-1)}}{i!}t^{iw} } \in R_{\mathbb{Q}}((s))_{\infty}.
$$

 \end{lemma}
 
 \begin{proof}
 Note that for the proof, we can replace $R$ with $R_{\mathbb{Q}}$ and  assume that $ R$ is a $\mathbb{Q}$-algebra.  If $1\leq a$ and $\alpha \in R((s)) $ then by the previous lemma we have $\sigma (\alpha)=\sum_{0 \leq i }\frac{(xt^w)^i\alpha ^{(i)}}{i!},$ which implies that 
 $$
 \sigma(e^{\alpha t^a})=e^{\sigma(\alpha)t^a}=e^{\sum_{0 \leq i }\frac{x^i\alpha ^{(i)}}{i!}t^{a+iw}}.
 $$

 On the other hand, for $f \in R((s))^{\times},$  both sides of the formula for $\sigma(f)$ above are compatible with multiplication.  Therefore, it is enough to prove the formula for $f =s,$ for $f \in R^{\times }$ or for $f \in 1+sR[[s]].$ The formula clearly holds for $f\in R^{\times}$ since $\sigma(f)=f$ and $f'=0.$ In case $f\in 1+sR[[s]],$ then letting  $\alpha:=\log (f) \in R[[s]],$ we have $f=e^{\alpha}.$ Using the previous lemma we obtain 
 $$
 \sigma (f)=e^{\sigma (\alpha)}=e^{\sum_{0 \leq i }\frac{x^i\alpha ^{(i)}}{i!}t^{iw}}=fe^{\sum_{1 \leq i }\frac{x^i(f'/f) ^{(i-1)}}{i!}t^{iw}}.
 $$
 
 For $f=s,$ we have 
 $$
 \sigma(s)=s+xt^w=s(1+\frac{xt^w}{s})=se^{\log(1+\frac{xt^w}{s})}=se^{\sum _{1\leq i}\frac{(-1)^{i-1}}{i}(\frac{xt^w}{s})^i}=se^{\sum _{1\leq i}\frac{x^i(1/s)^{(i-1)}}{i!}t^{iw}}.
 $$
 
\end{proof}
 
For a ring $R$ on which $(p-1)!$ is invertible and for $\alpha \in tR_{\infty},$  recall that we put $\underline{e}^{\alpha}:=\sum _{0\leq i<p}\frac{\alpha^i}{i!}.$ Next we consider the case when   $k$ is a field of  characteristic $p.$

 \begin{lemma}\label{lemma-repar-taylor2}
Let $\sigma$ be the automorphism of the $k_{\infty}$ algebra $k((s))_{\infty},$ which is the identity map modulo $(t)$ and has the property that  $\sigma(s)=s+x t^w,$ with $w \geq 1$ and $x\in k((s)),$ then $$\sigma(\underline{e}^{\alpha t^a})=\prod_{0\leq i<p}\underline{e}^{\frac{x^i
\alpha^{(i)}}{i!}t^{a+iw}}$$ in $k((s))_{p},$ for $1\leq a$ and $\alpha \in k((s));$ and 
$$\sigma(f)=f\prod_{1\leq i<p}\underline{e}^{\frac{x^i
(f'/f)^{(i-1)}}{i!}t^{iw}}$$ in $k((s))_{p},$ for $f \in k((s))^{\times}.$ 

 \end{lemma}
 \begin{proof}
 Suppose that we are given  $x,$ $f$ and $\alpha$ as in the statement of the lemma. 
 We can choose a ring $R$ without $\mathbb{Z}$-torsion and on which $(p-1)!$ is invertible such that there is a map $R\to k,$ and there is $\tilde{x,}$ $\tilde{f} \in  R((s))^{\times}$ and $\tilde{\alpha} \in  R((s)),$ which map to $x,$ $f$ and $\alpha.$ 
 We first apply the previous lemma to $\tilde{f}$ and $\underline{e}^{\tilde{\alpha}t^a},$ and the automorphism given by $\tilde{\sigma}(s)=s+\tilde{x}t^w.$ If we then  reduce the expression modulo $(t^p),$ notice that the expression is in $R((s))_p,$  and then  take the image under the map from $R$ to $k,$ the lemma follows.   \end{proof}
\begin{remark}
As in \textsection \ref{defn-of-omega-sec} above, if we slightly abuse the notation,   the images under $\sigma$   of $f \in R((s))^{\times}$ and of $e^{\alpha t^a}$ with $\alpha \in R((s)))$ and $1 \leq a$ can be expressed using a single formula. Namely, suppose that we allow the notation $f=e^{\alpha t^a},$  with $a=0.$ Then we can express both of  the formulas in Lemma \ref{lemma-repar-taylor1} as
$$
\sigma(e^{\alpha t^a})=e^{\sum _{0\leq i}\frac{x^i
\alpha^{(i)}}{i!}t^{a+iw}} \in R_{\mathbb{Q}}((s))_{\infty},
$$
for $a\geq 0.$ Note that if $a=0,$ this formula reads:  
$$
\sigma(f)=\sigma (e^{\alpha})=e^{\sum _{0 \leq i}\frac{x^i\alpha ^{(i)}}{i!}t^{iw}}=e^{\alpha}e^{\sum _{1 \leq i}\frac{x^i\alpha ^{(i)}}{i!}t^{iw}}=fe^{\sum _{1 \leq i}\frac{x^i(f'/f) ^{(i-1)}}{i!}t^{iw}},
$$
if we think of $\alpha $ as $\log (f)$ and hence use the convention that $\alpha ^{(i)}:=(f'/f)^{(i-1)}$ for $i \geq 1.$ 

In exactly the same manner, both of the formulas in Lemma \ref{lemma-repar-taylor2} can be expressed as 
$$\sigma(\underline{e}^{\alpha t^a})=\prod_{0\leq i<p}\underline{e}^{\frac{x^i
\alpha^{(i)}}{i!}t^{a+iw}} \in k((s))_{p},$$  for $0\leq a.$ 
\end{remark} 
 
 \subsection{The effect of reparametrization on $\Omega ^{(p)}$} 
In order to prove the invariance with respect to reparametrization, we will start with the lemma below which deals with the most basic infinitesimal automorphism. In fact, these automorphisms will generate all the automorphisms that we are interested in.  
\begin{lemma}\label{exact-lemma}
Suppose that $k$ is a field of characteristic $p,$ and $\sigma$ is the $k_{\infty}$-automorphism of $k((s))_{\infty}$ given by $\sigma(s)=s+xt^w,$ with $1\leq w$ and $x\in k((s)).$ Then 
\begin{align}\label{omegapsigma}
    \Omega ^{(p)}\big(\sigma(\underline{e}^{\alpha t^a}\wedge \underline{e}^{\beta t^b}\wedge \underline{e}^{\gamma t^c})\big)-\Omega ^{(p)}(\underline{e}^{\alpha t^a}\wedge \underline{e}^{\beta t^b}\wedge \underline{e}^{\gamma t^c}) \in \Omega ^{1} _{k((s))/k}
\end{align}
 is an exact form  for $a,b,c \geq 0$ and $\alpha, \beta , \gamma \in k((s)).$
\end{lemma}
  Note that in this case of characteristic $p$ dilogarithm, we do not need to assume that $w\geq 2$ as we had to in characteristic 0  \cite{un-higher}. 

Let us first express (\ref{omegapsigma}) using Lemma \ref{lemma-repar-taylor2} above as: 
\begin{align}\label{omegapsigma-alt}
\Omega ^{(p)}(\sum _{0 \leq i, j, k<p \atop {0<i+j+k}}\underline{e}^{\frac{x^i
\alpha^{(i)}}{i!}t^{a+iw}}\wedge \underline{e}^{\frac{x^j
\beta^{(j)}}{j!}t^{b+jw}}\wedge \underline{e}^{\frac{x^k
\gamma^{(k)}}{k!}t^{c+kw}}).    
\end{align}

Recall that $\Omega^{(p)}(\underline{e}^{ut^m}\wedge \underline{e}^{vt^n} \wedge \underline{e}^{wt^k})=0,$ if $m+n+k \neq p.$ 
This implies that  (\ref{omegapsigma-alt}) is 0, if $w \nmid p-(a+b+c)$ or if $a+b+c=p.$  So let us assume that $w | p-(a+b+c)$ and let $q=\frac{p-(a+b+c)}{w}>0.$  We will prove Lemma \ref{exact-lemma}  by proving the explicit identity:  
\begin{align}\label{antider}
 \Omega ^{(p)}(\sum _{0 \leq i, j, k<p \atop {0<i+j+k}}\underline{e}^{\frac{x^i
\alpha^{(i)}}{i!}t^{a+iw}}\wedge \underline{e}^{\frac{x^j
\beta^{(j)}}{j!}t^{b+jw}}\wedge \underline{e}^{\frac{x^k
\gamma^{(k)}}{k!}t^{c+kw}})= d(\sum _{i+j+k=q \atop {0 \leq i,j,k}} \frac{x^q}{q}S(a,b,c;i,j,k)\alpha^{(i)}\beta ^{(j)}\gamma ^{(k)}),  
\end{align}
where $S(a,b,c;i,j,k)$ is defined as follows. We define 
$$
S(a, b, c;i,j,k):=\frac{bk-cj}{i!j!k!},
$$
if $a+iw>{\rm max}\{b+jw,c+kw \}$ or $b+jw=c+kw>a+iw. $ 
 
In case, $b+jw >{\rm max}\{a+iw,c+kw \}$ or $a+iw=c+kw>b+jw $ then we let 
$$
S(a, b, c;i,j,k):=-S(b,a,c;j,i,k).
$$

In case, $c+kw >{\rm max}\{a+iw,b+jw \}$ or $a+iw=b+jw>c+kw $ then we let
$$
S(a, b, c;i,j,k):=S(c,a,b;k,i,j).
$$
Since $i+j+k=q,$ we have $a+iw+b+jw+c+kw=p.$ This implies that  $a+iw=b+jw=c+kw$ is impossible since $p$ is a prime greater than 3.

Let us explain how the expression on the right hand side of (\ref{antider}) in fact makes sense.   Note that the expression  involves terms of the form $\alpha ^{(i)}\beta^{(j)}\gamma^{(k)}.$ On the other hand for $k=0,$ we defined $\gamma ^{(k)}$ only for $c>0.$ Therefore, we have to make sure that the coefficient $S(a,b,c;i,j,k)$ is 0, if $k=c=0.$ The same statement would be true for $j=b=0$ and $i=a=0,$ since  
\begin{align}\label{anti-symmetry-def}
    S(\sigma_{1}(a),\sigma_{1}(b),\sigma_{1}(c);\sigma_{2}(i),\sigma_{2}(j),\sigma_{2}(k))={\rm sign}(\sigma)S(a,b,c;i.j,k),
\end{align}
for $\sigma$ in the group of permutations of $\{(a,i), (b,j),(c,k) \}$ and $\sigma_{i}$ is the $i$-th component of $\sigma,$ for $i=1,\,2.$  If $k=c=0,$ then $(a+iw)+(b+jw)=p.$ Since $p>2,$ we then have $a+iw \neq b+jw.$ In case,  $a+iw>b+jw$ then $S(a,b,c;i,j,k)=\frac{bk-cj}{i!j!k!}=0.$ In case, $b+jw>a+iw,$ the same follows again by (\ref{anti-symmetry-def}).

\subsection{The proof of Lemma \ref{exact-lemma}}

We will instead prove the identity (\ref{antider}), which gives a more precise statement than the lemma.
\subsubsection{The coefficients of $x^{q-1}x'\alpha^{(i)}\beta^{(j)}\gamma^{(k)}$ } Suppose that $i+j+k=q.$ The coefficient of  $x^{q-1}x'\alpha^{(i)}\beta^{(j)}\gamma^{(k)}$ in the right side of  (\ref{antider}) is $S(a,b,c;i,j,k).$ We need to check that this is the same as its coefficient in the left side of (\ref{antider}). 

Suppose first that $a+iw>{\rm max}\{b+jw,c+kw \}$ or $b+jw=c+kw>a+iw. $  The term in (\ref{omegapsigma}) that contributes to this coefficient is 
$$
\frac{x^i}{i!}\alpha^{(i)}((b+jw)\frac{x^j}{j!}\beta^{(j)}\frac{kx^{k-1}x'}{k!}\gamma^{(k)}-(c+kw)\frac{x^k}{k!}\gamma^{(k)}\frac{jx^{j-1}x'}{j!}\beta^{(j)})=x^{q-1}x'\alpha^{(i)}\beta^{(j)}\gamma^{(k)}\frac{bk-cj}{i!j!k!},
$$
whose coefficient is precisely 
$S(a,b,c;i,j,k).$ The other cases follow from this one by the anti-symmetry of $\Omega^{(p)}$ and $S$.

\subsubsection{The coefficients of  $\frac{x^{q}}{i! j! k!}\alpha^{(i)}\beta^{(j)}\gamma^{(k)}$} We need to check that  on both sides of (\ref{antider}) the coefficients of $\frac{x^{q}}{i! j! k!}\alpha^{(i)}\beta^{(j)}\gamma^{(k)}$ are the same.  Let us temporarily write $\tilde{a}:=a+iw,$ $\tilde{b}:=b+jw,$ and $\tilde{c}:=c+kw$ with $i+j+k=q+1$ and let $\tilde{a}':=\tilde{a}-w,$ $\tilde{b}':=\tilde{b}-w$  and $\tilde{c}':=\tilde{c}-w.$ There are many cases to consider and we will only write down the answer in each case. The computations are routine but tedious and are omitted. Without loss of generality, we will assume that $\tilde{a} \geq \tilde{b} \geq \tilde{c}.$ 

{\bf Case (i).} If one of the following cases hold: 
$$
\tilde{a} > \tilde{a}'>\tilde{b} > \tilde{b}'>\tilde{c}>\tilde{c}',\;\;\; \tilde{a} > \tilde{a}'>\tilde{b} > \tilde{c}>\tilde{b}'>\tilde{c}',\;\;\; \tilde{a} > \tilde{a}'>\tilde{b} > \tilde{b}'=\tilde{c}>\tilde{c}',
$$
$$
 \tilde{a} > \tilde{a}'>\tilde{b} = \tilde{c}>\tilde{b}'=\tilde{c}',\;\;\; \tilde{a} > \tilde{b}=\tilde{c} > \tilde{a}'>\tilde{b}'=\tilde{c}'
$$
then the coefficient on the left side of (\ref{antider}) is 
$$
i\cdot 0+(-1)j\cdot (c+kw)+k\cdot (b+jw)=kb-jc
$$
and the coefficient 
on the right side of (\ref{antider}) is 
$$
\frac{1}{q}(i(kb-jc)+j(kb-(j-1)c)+k((k-1)b)-jc)=\frac{i+j+k-1}{q}(kb-jc)=kb-jc.
$$
Therefore the coefficients on both sides of the identity match. Note that the assumption that the characteristic is $p$ is not needed in this case. 

{\bf Case (ii).} If one of the following cases hold: 
$$
\tilde{a} > \tilde{b}>\tilde{c} > \tilde{a}'>\tilde{b}'>\tilde{c}',\;\;\; \tilde{a} > \tilde{b}>\tilde{a}' > \tilde{c}>\tilde{b}'>\tilde{c}',\;\;\; \tilde{a} > \tilde{b}>\tilde{a}' > \tilde{b}'>\tilde{c}>\tilde{c}',
$$
$$
 \tilde{a} > \tilde{b}>\tilde{a}' > \tilde{b}'=\tilde{c}>\tilde{c}',\;\;\; \tilde{a} > \tilde{b}>\tilde{c} = \tilde{a}'>\tilde{b}'>\tilde{c}'
$$
then the coefficient on the left side of (\ref{antider}) is 
$$
i(c+kw)+(-1)j\cdot (c+kw)+k\cdot (b+jw)=ic+ikw+bk-jc.
$$
Putting first $i=q+1-j-k$ and then noting that $qw=-a-b-c$ in characteristic $p$ this can be rewritten as 
$$
qc+c-2jc-2kc-ka+kw-jkw-k^2w.
$$

The coefficient on the right side of (\ref{antider}) is 
$$
\frac{1}{q}(i((i-1)c-ka)-j((j-1)c-kb)+k((k-1)b-jc)).
$$
Again first putting $q+1-j-k$ instead of $i$ and then using $-w=\frac{a+b+c}{q},$ we see that this expression matches the one above.

{\bf Case (iii).} If one of the following cases hold: 
$$
\tilde{a} = \tilde{b}>\tilde{a}' = \tilde{b}'>\tilde{c}>\tilde{c}',\;\;\; \tilde{a} = \tilde{b}>\tilde{c} > \tilde{a}'=\tilde{b}'>\tilde{c}',\;\;\; \tilde{a} = \tilde{b}>\tilde{a}' = \tilde{b}'=\tilde{c}>\tilde{c}',
$$
then the coefficient on the left side of (\ref{antider}) is 
$$
-i(-(c+kw))+j( -(c+kw))+k 0=ic-jc+ikw-jkw.
$$
Proceeding exactly as in the previous case, this can be rewritten as 
$$
qc+c-2jc-2kc-ka-kb+kw-2jkw-k^2w.
$$

The coefficient on the right side of (\ref{antider})  is 
$$
\frac{1}{q}(-i(ka-(i-1)c)+j(kb-(j-1)c)+k(ja-ib)).
$$
This matches the above expression after first we put $i=q+1-j-k,$ then replace  $a+b+c$ with $-qw.$

{\bf Case (iv).} If one of the following cases hold: 
$$
\tilde{a} > \tilde{a}'=\tilde{b} > \tilde{c}>\tilde{b}'>\tilde{c}',\;\;\; \tilde{a} > \tilde{a}'=\tilde{b} > \tilde{b}'>\tilde{c}>\tilde{c}',\;\;\; \tilde{a} > \tilde{a}'=\tilde{b} > \tilde{b}'=\tilde{c}>\tilde{c}',
$$
then the coefficient on the left side of (\ref{antider}) is
$$
-i(b+jw)-j( c+kw)+k (b+jw)=-2ib-jc+qb+b-jb-ijw.
$$

The coefficient on the right side of (\ref{antider}) is 
$$
\frac{1}{q}(i(ja-(i-1)b)-j((j-1)c-kb)+k((k-1)b-jc)).
$$
Replacing $k$ with $q+1-i-j$ and proceeding as above we see that this matches the above expression. 

{\bf Case (v).} If  
$
\tilde{a} = \tilde{b}=\tilde{c} > \tilde{a}'=\tilde{b}'=\tilde{c}'$ then  the coefficient on the left side of (\ref{antider}) is 0. In this case, the coefficient on the right side of (\ref{antider}) is 
$$
\frac{1}{q}(i(kb-jc)-j(ka-ic)+k(ja-ib))=0.
$$

Finally, let us show that the above cases cover all possibilities. First, note that with our notation, we have 
$$
\tilde{a}'+\tilde{b}+\tilde{c}=\tilde{a}+\tilde{b}'+\tilde{c}=\tilde{a}+\tilde{b}+\tilde{c}'=p.
$$ 
Therefore, the cases $\tilde{a}'=\tilde{b}=\tilde{c}$ or $\tilde{a}=\tilde{b}'=\tilde{c}$ or $\tilde{a}=\tilde{b}=\tilde{c}'$ are eliminated since $p$ is a prime greater than 3. Note that by our assumptions $\tilde{a} \geq \tilde{b} \geq \tilde{c},$  $\tilde{a}> \tilde{a}',$  $\tilde{b}> \tilde{b}',$ and $\tilde{c}> \tilde{c}'.$  

(a) if $\tilde{a}=\tilde{b}=\tilde{c},$ then this is covered by case (v). 

(b) if $\tilde{a}=\tilde{b}>\tilde{c},$ then this is covered by case (iii)

(c)  if $\tilde{a}>\tilde{b}=\tilde{c},$ then either $\tilde{b}=\tilde{c}>\tilde{a}'$ or $\tilde{a}'>\tilde{b}=\tilde{c}$ and both possibilities are covered by case (i).

For the remainder of the discussion we assume that $\tilde{a} > \tilde{b} > \tilde{c}.$ The possibilities depend on the location of $\tilde{a}'$ and $\tilde{b}'.$ Note that $\tilde{a}'>\tilde{b}'>\tilde{c}'.$ 

(d) if $\tilde{c}\geq\tilde{a}',$ then this is covered by case (ii)

(e) if $\tilde{b}>\tilde{a}'>\tilde{c},$ then either $\tilde{c}> \tilde{b}'$  or $\tilde{c}= \tilde{b}'$ or $\tilde{b}'> \tilde{c}$ and these possibilities are  covered by case (ii)
 
(f) if $\tilde{b}=\tilde{a}'$ then either $\tilde{c}>\tilde{b}'$ or $\tilde{c}=\tilde{b}'$ or $\tilde{b}'>\tilde{c}$ and all  possibilities are covered by case  (iv)

(g) if $\tilde{a}'>\tilde{b}$ then either  $\tilde{c}>\tilde{b}'$ or $\tilde{c}=\tilde{b}'$ or $\tilde{b}'>\tilde{c}$ and and all  possibilities are covered by case  (i).
\hfill $\Box$

 \subsection{The invariance of the residues of $\Omega ^{(p)}$} 
 We will next prove the invariance with respect to reparametrization. 
 
 \begin{proposition}\label{invariance-prop}
 Suppose that $k$ is a field of characteristic $p,$ and $\sigma$ is a continuous $k_{\infty}$-automorphism of $k((s))_{\infty}$ which is identity modulo the ideal $(t).$ Then we have the equality of the residues: 
 $$
 res_{s=0} \Omega ^{(p)}(\sigma (\underline{e}^{\alpha t^a}\wedge \underline{e}^{\beta  t^b}\wedge \underline{e}^{\gamma t^c}) )=res_{s=0} \Omega ^{(p)}( \underline{e}^{\alpha t^a}\wedge \underline{e}^{\beta  t^b}\wedge \underline{e}^{\gamma t^c} ) \in k,
 $$
 for $a,b,c \geq 0$ and $\alpha, \beta , \gamma \in k((s)).$
 \end{proposition}
 
 \begin{proof}
Let us write $\sigma(s)=s+\sum_{1\leq w}x_w t^w,$ with $x_{w} \in k((s))).$ If we let $\tau$ be the automorphism of $k((s))_{\infty}$ given by $\tau(s)=s+\sum _{1\leq w<p}x_{w}t^w,$ since $\sigma$ and $\tau $ are congruent modulo $(t^p),$ 
$$
\Omega ^{(p)}(\sigma(\underline{e}^{\alpha t^a}\wedge \underline{e}^{\beta  t^b}\wedge \underline{e}^{\gamma t^c}))=\Omega ^{(p)}(\tau(\underline{e}^{\alpha t^a}\wedge \underline{e}^{\beta  t^b}\wedge \underline{e}^{\gamma t^c})) .
$$
Letting similarly $\tau_{w}$ defined by $\tau_w(s)=s+y_{w}t^w$ for $1 \leq w<p,$ we notice that $\tau_{w}$ are of the form as in the statement of Lemma \ref{exact-lemma}. This lemma then implies that for any $q \in \Lambda ^{3} k((s))_{\infty} ^{\times},$ 
$$
res_{s=0}(\Omega^{(p)}(\tau_{w}(q)))=res_{s=0}(\Omega ^{(p)}(q)).
$$
Choosing $y_{w} \in k((s))$ appropriately for $1\leq w<p,$ we can write  $\tau=\tau_{p-1}\circ \cdots \tau_{2}\circ \tau_{1}$ and applying the last equality $(p-1)$-times gives us the equality: 
$$
res_{s=0} \Omega ^{(p)}(\tau (\underline{e}^{\alpha t^a}\wedge \underline{e}^{\beta  t^b}\wedge \underline{e}^{\gamma t^c}) )=res_{s=0} \Omega ^{(p)}( \underline{e}^{\alpha t^a}\wedge \underline{e}^{\beta  t^b}\wedge \underline{e}^{\gamma t^c} ),
$$
from which the proposition follows. 
\end{proof}

 \subsection{The definition of the residue of  $ {\omega}^{(p)}$ for a  pair of  smooth algebras of dimension 1 over $k_p$}\label{section definition omega}
Using the previous section, we generalize the definition of $\omega^{(p)}$ to pairs of elements in smooth algebras of dimension 1 over $k_p$ with the same reduction. As opposed to the split algebra case, which was considered in  \textsection  \ref{split-def}, the general case will essentially depend on   Proposition \ref{invariance-prop} above. 

First, we note that if $\pazocal{R}$ is a smooth $k_{n}$-algebra then it is (non-canonically) split. Even though this is standard, we provide a proof since we could not find an appropriate reference. 
\begin{lemma}\label{smooth-split}
Suppose that $\pazocal{R}$ is a smooth $k_{n}$-algebra and $\underline{\pazocal{R}}:=\pazocal{R}/(t).$  There is a, not necessarily unique, $k_{n}$-algebra isomorphism from $\pazocal{R}$ to $\underline{\pazocal{R}}_{n},$ which reduces to the identity map modulo $(t).$ 
\end{lemma}

\begin{proof}
Since $\pazocal{R}$ is smooth over $k_n,$ it is formally smooth. This implies that since $(t)$ is a nilpotent ideal  in $\underline{\pazocal{R}}_n$ and $\underline{\pazocal{R}}=\underline{\pazocal{R}}_n/(t),$ a map $f$ exists in the following  diagram 
\[
\begin{tikzcd}
\pazocal{R} \arrow[dr, dashed, "f"]  \arrow[r] &\underline{\pazocal{R}}  \\
k_{n} \arrow[u] \arrow[r] &\underline{\pazocal{R}}_n \arrow[u] 
\end{tikzcd}
\]
that makes it commute. This commutativity implies that $f(t)=t$ and the reduction of $f$ modulo $(t)$ is the identity map. Let us show that $f$ has to be an isomorphism. 

Assume that $f$ is an isomorphism modulo $(t^i)$ for $1\leq i < n.$ If $f(r) \in (t^{i+1})$ then by assumption $r \in (t^i).$ Let us write $r=st^{i}$ with $s \in \pazocal {R}.$ Then $f(s)t^i \in (t^{i+1})$ implies that $f(s) \in (t).$ This in turn implies that $s$ is in  $(t),$ since $f$ is an isomorphism modulo $(t).$ Combining, we obtain $r=st^i \in (t^{i+1}).$ Therefore, $f$ modulo $(t^{i+1})$ is injective. 

Let us show that $f$ modulo $(t^{i+1})$ is surjective.  Given $b \in \underline{\pazocal{R}}_n,$ by assumption, there is $a \in \pazocal{R}$ such that $b-f(a) \in (t^i).$ Let us put $b-f(a)=st^i,$ with $s \in \underline{\pazocal{R}}_n.$ Since $f$ is an isomorphism modulo $(t),$ there is a $\varepsilon \in \pazocal{R}$ such tat $s-f(\varepsilon) \in (t).$ Combining we obtain that $b-f(a+\varepsilon t^i) \in (t^{i+1}).$ 

Together, these show that $f$ is an isomorphism modulo $(t^{i+1}).$ The result follows by induction on $i.$ 
\end{proof}

Let $\pazocal{R}$ be a smooth $k_p$-algebra of relative dimension 1 as above. Let $\eta$ be the generic point and $x$ be a closed point of the spectrum of $\pazocal{R}.$ Then we define a  map 
$$
res_{x}\omega ^{(p)}:\Lambda ^3 (\pazocal{R}_{\eta},(t))^{\times}\to k',
$$
where $k'$ is the residue field of $x,$ as follows. To ease the notation let us write $\pazocal{S}=\pazocal{R}_{\eta}.$  Since $\pazocal{R}/k_p$ is smooth, there is an isomorphism $\varphi:\pazocal{S} \to \underline{\pazocal{S}}_p$ of $k_p$-algebras which is the identity map  modulo $(t)$ by Lemma \ref{smooth-split}. Here $\underline{\pazocal{S}}$ denotes the reduction of $\pazocal{S}$ modulo $(t)$ and $\underline{\pazocal{S}}_p:=\underline{\pazocal{S}}[t]/(t^p).$ 

The map $\varphi$ then induces an isomorphism 
$$
\Lambda ^{3}\varphi: \Lambda ^3 (\pazocal{S},(t))^{\times} \to \Lambda ^3 (\underline{\pazocal{S}}_p,(t))^{\times},
$$
by transport of structure, 
whose composition with 
$$
\tilde{\Omega }^{(p)}:\Lambda ^3 (\underline{\pazocal{S}}_p,(t))^{\times}\to \Omega ^{1} _{\underline{\pazocal{S}}},
$$
will be denoted by $\omega^{(p)} _{\varphi}: \Lambda ^3 (\pazocal{S},(t))^{\times} \to \Omega ^{1} _{\underline{\pazocal{S}}}.$ Taking the projection to $\Omega^{1}_{\underline{S}/k}$ and taking residue at the closed point $x$ gives us the map $res_{x}\omega^{(p)} _{\varphi}.$ The next lemma, which shows that the residue is independent of the choice of $\varphi,$ is essential. 

\begin{lemma}\label{lemma-indep-res}
If $\varphi$ and $\psi$ are two isomorphisms between the $k_p$-algebras  $\pazocal{S}$ and $\underline{S}_{p}$ then $$res_{x}\omega^{(p)} _{\varphi}=res_{x}\omega^{(p)} _{\psi}.$$ 
\end{lemma}

\begin{proof}
Let $\underline{\hat{\pazocal{S}}}$ denote the completion of $\underline{S} $ at the valuation determined by $x.$ Then $\underline{\hat{\pazocal{S}}}$ is the field of fractions of the discrete valuation ring $(\underline{\hat{\pazocal{S}}})^{\circ}$ obtained by completing $\underline{\pazocal{R}}$ at $x.$ Since $(\underline{\hat{\pazocal{S}}})^{\circ}$ contains a field, the structure theorem for complete local rings \cite[Theorem 28.3, \textsection 29]{mat} imply that, if $k'$ is the residue field of $(\underline{\hat{\pazocal{S}}})^{\circ}$ and $s$ is a uniformizer, then there is a surjection from $k'[[s]] $ to $(\underline{\hat{\pazocal{S}}})^{\circ}.$ The kernel of this map has to be 0 since the quotient does not have nilpotent elements. Passing to the field of fractions, we obtain a  $k$-isomorphism $k'((s))) \simeq \underline{\hat{\pazocal{S}}} .$ By Lemma \ref{smooth-split}, we also know that there is a $k_p$-isomorphism $\pazocal{S}\simeq \underline{\pazocal{S}}_p$ which is identity modulo $(t).$ Therefore, replacing $\underline{\pazocal{S}}$ by $\underline{\pazocal{S}}_p$ and then completing at $x,$ we are reduced to proving the lemma for $\pazocal{S}=k'((s))_p$ and $\varphi$ and $\psi$ two $k_p$-automorphisms of $k'((s))_{p}$ which reduce to the identity automorphism modulo $(t).$ The smoothness assumption implies that  $k'/k$ is separable. Since $\varphi$ and $\psi$ are $k$-algebra morphisms which reduce to identity modulo $(t),$ the separability of $k'/k$ implies that  they are  also $k'$-algebra morphisms. 

Suppose that $(\tilde{f},\hat{f}), (\tilde{g}, \hat{g}), (\tilde{h},\hat{h} ) \in (k'((s))_{p},(t))^{\times}.$ Then  $res_{x}\omega^{(p)} _{\varphi}((\tilde{f},\hat{f})\wedge(\tilde{g}, \hat{g})\wedge(\tilde{h},\hat{h} ))$ is equal to the residue of 
\begin{align}\label{diff-res-form}
\Omega ^{(p)}(\varphi(\tilde{f})\wedge \varphi(\tilde{g}) \wedge \varphi(\tilde{h}) )-\Omega ^{(p)}(\varphi(\hat {f})\wedge \varphi(\hat{g}) \wedge \varphi(\hat{h}) )
\end{align}
at $s=0.$ Note that the proof of  Proposition \ref{invariance-prop} works without modification when $\sigma$ is a continuous $k_p$-automorphism of $k((s))_p,$ which is identity modulo $(t).$ Since  $\varphi$ is obtained by completing an automorphism at $x,$ it is continuous and hence satisfies the hypotheses of Proposition \ref{invariance-prop}. Therefore,   the residue of
\begin{align}\label{residue-same}
    \Omega ^{(p)}(\tilde{f}\wedge \tilde{g} \wedge \tilde{h} )-\Omega ^{(p)}(\varphi(\tilde{f})\wedge \varphi(\tilde{g}) \wedge \varphi(\tilde{h}) )
\end{align}
at $s=0,$ is equal to $0.$ Since a similar formula is true for the second term in (\ref{diff-res-form}), we can rewrite the residue of this expression as the residue of 
$$
\Omega ^{(p)}(\tilde{f}\wedge \tilde{g} \wedge \tilde{h} )-\Omega ^{(p)}(\hat{f}\wedge \hat{g} \wedge \hat{h} ),
$$
which of course does not depend on $\varphi.$ 
\end{proof}

Since, as we have seen, $res_{x}\omega^{(p)} _{\varphi}$ does not depend on $\varphi,$ we will denote it by $res_{x}\omega^{(p)} .$ Below, we will need a variant of this construction for pairs of algebras. Next we describe this.

Suppose that  $A$ is a ring with an ideal $I$ and $B$ and $B'$ are two $A$-algebras together with an isomorphism $\chi: B/IB \simeq B'/IB'$ of $A$-algebras. We let 
$$
(B,B',\chi)^{\times}:=\{(p,p')| p \in B^{\times} \;{\rm and} \;p' \in B'^{\times} \; {\rm s.t.}\; \chi(p|_I)=p'|_I \},
$$
where $p|_{I}$ denotes the image of $p$ in $(B/IB)^{\times}.$ 

Suppose that $\pazocal{R}$ and $\pazocal{R}'$ are  smooth $k_{p}$-algebras of relative dimension 1, together with a $k$-isomorphism $\chi:\underline{\pazocal{R}} \to \underline{\pazocal{R}}'.$ We identify the spectra of $\pazocal{R}$ and $\pazocal{R}'$ via this isomorphism. Suppose that  $x$ is a closed point and $\eta$ is the generic point of this spectrum. We would like to construct a map 
$$
res_{x}\omega^{(p)} _{\chi_{\eta}}:\Lambda ^3 (\pazocal{R}_{\eta},\pazocal{R}'_{\eta},\chi_{\eta})^{\times} \to k',
$$
where $k'$ is the residue field of $x.$ Let us put, as above, $\pazocal{S}:=\pazocal{R}_{\eta}$ and $\pazocal{S}':=\pazocal{R}'_{\eta}.$ Since, as above, $\pazocal{S}\simeq \underline{\pazocal{S}}_{p}$ and $\pazocal{S}'\simeq \underline{\pazocal{S}}'_{p}$ we can choose an isomorphism $\tilde{\chi}_{\eta}: \pazocal{S} \to \pazocal{S}'$ which lifts $\chi_{\eta}.$ The map $\tilde{\chi}_{\eta},$ induces an isomorphism
$$
\xymatrix{
(\pazocal{S},\pazocal{S}',\chi_{\eta})^{\times}  \ar^{\tilde{\chi}_{\eta}^{*}}[r] & (\pazocal{S},(t))^{\times}}.
$$
Composing these, we obtain 
$$
\xymatrix{
res_{x}\omega^{(p)} _{\tilde{\chi}_{\eta}}:\Lambda ^3 (\pazocal{R}_{\eta},\pazocal{R}'_{\eta},\chi_{\eta})^{\times} \ar^{\;\;\;\;\;\;\;\;\;\;\;\;\;\;\Lambda^3\tilde{\chi}_{\eta}^{*}}[r]& \Lambda ^{3} (\pazocal{R}_{\eta},(t))^{\times}  \ar^{\;\;\;\;\;\;\;\;\;\;\;\;res_x\tilde{\Omega}^{(p)}}[r]& k'.}
$$

\begin{lemma}\label{lemma-indep-reduced}
The map $res_{x}\omega^{(p)} _{\tilde{\chi}_{\eta}}$ does not depend on the choice of the lifting $\tilde{\chi}_{\eta}$ of $\chi_{\eta}.$
\end{lemma}

\begin{proof}
The proof is exactly the same as that of Lemma \ref{lemma-indep-res}. We reduce to the case where $\pazocal{R}_{\eta}$ and $\pazocal{R}'_{\eta}$ are both $k'((s))_{p}$ and $\tilde{\chi}_{\eta}$ is replaced with $\varphi$ in the proof of this lemma.   Then the result follows from the fact that the residue of (\ref{residue-same}) at $s=0$ is 0. 
\end{proof}

\noindent Since this map is independent  of the choice of the lifting $\tilde{\chi}_{\eta},$ we will denote it by $res_{x}\omega^{(p)} _{\chi_{\eta}}.$ If the map $\chi_{\eta}$ is clear from the context, we will simply denote it by $res_{x}\omega^{(p)}.$

The following lemma is an immediate consequence of the definition of $res_{x}\omega^{(p)}.$ 
\begin{lemma}\label{lemma-sum-residue}
Suppose that $\pazocal{R},$ $\pazocal{R}'$ and $\pazocal{R}''$ are smooth $k_{p}$-algebras of relative dimension 1 and $\phi:\underline{\pazocal{R}} \to \underline{\pazocal{R}}'$ and $\psi:\underline{\pazocal{R}}' \to \underline{\pazocal{R}}''$ are $k$-algebra isomorphisms of their reductions modulo $(t).$ Following the notation above, if $(\mathfrak{q},\mathfrak{q}') \in \Lambda ^{3}(\pazocal{R}_{\eta},\pazocal{R}'_{\eta}, \phi_{\eta} )^{\times}$ and  $(\mathfrak{q}',\mathfrak{q}'') \in \Lambda ^{3}(\pazocal{R}'_{\eta},\pazocal{R}''_{\eta}, \psi_{\eta} )^{\times}$ then 
$$
(\mathfrak{q},\mathfrak{q}'') \in \Lambda ^{3}(\pazocal{R}_{\eta},\pazocal{R}''_{\eta}, \psi_{\eta}\circ \phi_{\eta} )^{\times}
$$
and 
$$
res_{x}\omega^{(p)}(\mathfrak{q},\mathfrak{q}'')=
res_{x}\omega^{(p)}(\mathfrak{q},\mathfrak{q}')+ res_{x}\omega^{(p)}(\mathfrak{q}',\mathfrak{q}'').
$$
\end{lemma}

\section{The residue of $\Omega^{(p)}$ in the case of good liftings} 

In this section, we relate the residues of $\Omega ^{(p)}$ to $\ell i_{2}^{(p)}$ in the case of good liftings, continuing with our assumption that $p \geq 5. $ We start wıth comparing  the difference between the values under $\ell^{(p)}$ of the residues of  
$$
\mathfrak{q}'=(s-xt^w)\wedge \underline{e}^{\alpha t^a}\wedge \underline{e}^{\beta t^b} \in \Lambda ^{3} k((s))_{p} ^{\times}
$$
and 

$$
\mathfrak{q}=s\wedge \underline{e}^{\alpha t^a}\wedge \underline{e}^{\beta t^b}  \in \Lambda ^{3} k((s))_{p} ^{\times},
$$
where $w \geq 2, $ $a,b\geq 0,$ $x \in k$ and $\alpha, \, \beta \in k[[s]].
$
If $a=0$ or $b=0,$ we follow the conventions above.

\begin{lemma}\label{lemma-special-residue-formula}
With notation as above, we have
\begin{align}
 \ell^{(p)}(res_{s-xt^w}(\mathfrak{q}'))-\ell^{(p)} (res_{s}(\mathfrak{q}))=res_{s=0} \Omega ^{(p)}(\mathfrak{q}'-\mathfrak{q}).
\end{align}
\end{lemma}

\begin{proof} We start with noting that 
$$
res_{s-xt^w}(\mathfrak{q}')=\underline{e}^{\alpha (xt^w)t^a}\wedge \underline{e}^{\beta(xt^w)t^b}
$$
and $res_{s}(\mathfrak{q})=\underline{e}^{\alpha (0)t^a}\wedge \underline{e}^{\beta(0)t^b}.$ This follows from the definition of the residue map (\ref{res-lam}) together with the observation that if $f(t),\, g(t)  \in k[[t]],$ with $f(0)=0,$  and $h(s) \in k[[s]]$ then under the isomorphism $$k[[s,t]]/(s-f(t)) \simeq k[[t]],$$ $h(s)g(t)$ maps to $h(f(t))g(t).$

Recall that $\ell^{(p)}=\frac{1}{2} \sum _{1 \leq i <p }i\cdot \ell _{p-i} \wedge \ell _{i}.$ Therefore,  if $0 \leq i,\, j$ and $y, \, z \in k,$ then
$$
\ell^{(p)}(\underline{e}^{yt^j}\wedge \underline {e}^{zt^i})=0,
$$
if $i=0, $ or $j=0$ or $i+j \neq p;$ and 
$$
\ell^{(p)}(\underline{e}^{yt^j}\wedge \underline {e}^{zt^i})=iyz,
$$
if $0<i, \, j$ and $i+j=p.$ 

First, suppose that $w \nmid p-(a+b)$ or $p<a+b.$ In this case,  since  $res_{s}(\mathfrak{q})=\underline{e}^{\alpha (0)t^a}\wedge \underline{e}^{\beta(0)t^b}$ with $a+b \neq 0, $ we have 
$$
\ell^{(p)} (res_{s}(\mathfrak{q}))=0.
$$
Since $\alpha (xt^w)t^a,$   has monomials only of type $t^j$ with $a\leq j$ and $w|(j-a)$ and $\beta (yt^w)t^b,$  has monomials only of type $t^i$ with $b\leq i$ and $w|(i-b),$ $
res_{s-xt^w}(\mathfrak{q}')
$ 
is a sum of terms of the form $\underline{e}^{yt^j}\wedge \underline {e}^{zt^i},$ with $a+b \leq i+j$ and  $w\nmid (i+j -(a+b)).$ If $p<a+b,$ the first condition implies that $p<i+j;$ if $w \nmid p-(a+b),$  then  $w\nmid (i+j -(a+b))$ implies that $i+j \neq p.$ Therefore, in both cases we have $\ell ^{(p)}(\underline{e}^{yt^j}\wedge \underline {e}^{zt^i})=0.$ Summing these, we obtain that 
$$
\ell^{(p)}(res_{s-xt^w}(\mathfrak{q}'))=0.
$$ 
On the other hand, 
$$
\Omega^{(p)}(\mathfrak{q}'-\mathfrak{q})=\Omega ^{(p)}(\underline{e}^{-\sum _{1\leq n<p}\frac{1}{n}(\frac{xt^w}{s})^n}\wedge \underline{e}^{\alpha t^a}\wedge \underline{e}^{\beta t^b}).
$$
The expression in the last parantheses is a sum of elements of the form $\underline{e}^{\gamma t^c}\wedge \underline{e}^{\alpha t^a}\wedge \underline{e}^{\beta t^b},$ with $0<c$ and $w|c.$ If $p<a+b$ then $p<a+b+c;$ if $w \nmid p-(a+b)$ then $a+b+c \neq p.$ In each of these cases $\Omega^{(p)}(\underline{e}^{\gamma t^c}\wedge \underline{e}^{\alpha t^a}\wedge \underline{e}^{\beta t^b})=0.$ Summing these terms, we obtain that $\Omega^{(p)}(\mathfrak{q}'-\mathfrak{q})=0.$ 

Now suppose that $a+b=p.$ In 
$$
\ell^{(p)} (res_{s}(\mathfrak{q}))=\frac{1}{2} \sum _{1 \leq i <p }i\cdot (\ell _{p-i} \wedge \ell _{i} )(\underline{e}^{\alpha (0)t^a}\wedge \underline{e}^{\beta(0)t^b})
$$
the only terms  that contribute are  the terms corresponding to $i=a$ and $i=b.$ These contributions sum, to give, 
$$
\ell^{(p)} (res_{s}(\mathfrak{q}))=b\alpha(0)\beta(0),
$$
since $a=-b$ in $k.$ The only contribution to 
$$
\ell^{(p)}(res_{s-xt^w}(\mathfrak{q}'))=\ell^{(p)}(\underline{e}^{\alpha (xt^w)t^a}\wedge \underline{e}^{\beta(xt^w)t^b})
$$
comes from the constant terms of $\alpha(xt^e)$ and $\beta (xt^w).$ Therefore, 
$$
\ell^{(p)}(res_{s-xt^w}(\mathfrak{q}'))=\ell ^{(p)}(\underline{e}^{\alpha (0)t^a}\wedge \underline{e}^{\beta(0)t^b})=b\alpha (0) \beta(0). 
$$
Finally, as above $\mathfrak{q}'-\mathfrak{q}$ is a sum of terms of the form $\underline{e}^{\gamma t^c}\wedge \underline{e}^{\alpha t^a}\wedge \underline{e}^{\beta t^b},$ all of which have the property that $a+b+c >p.$ This implies that  $\Omega^{(p)}(\mathfrak{q}'-\mathfrak{q})=0.$

In case, $a=b=0$ and $w=p,$ we similarly have $\ell^{(p)}(res_{s-xt^w}(\mathfrak{q}'))=\ell^{(p)} (res_{s}(\mathfrak{q}))=0$ and $\Omega^{(p)}(\mathfrak{q}'-\mathfrak{q})=0.$ 

Therefore, from now on, we assume  that $w|(p-(a+b))$ and $0<a+b<p$ and we put $q=\frac{p-(a+b)}{w}.$ We also put $\alpha=\sum _{0\leq i}\alpha _i s^i$ and $\beta=\sum _{0\leq i}\beta _i s^i.$   
With these assumptions,  $\ell^{(p)}(res_{s}(\mathfrak{q}))=0$ and $\ell^{(p)}(res_{s-xt^w}(\mathfrak{q}'))=x^q\sum _{i+j=q}(b+wj)\alpha _i \beta_j.$ In order to compute $res_{s=0}\Omega ^{(p)}(\mathfrak{q}'-\mathfrak{q}),$ we look at the possible cases:

(i) ${\rm max}\{a,b \}< p-(a+b)$

(ii) $p-(a+b)\leq{\rm min}\{a,b \}  $ and $a \neq b.$ 

(iii) $p-(a+b)<a  $ and $a = b.$

(iv) $p-(a+b)={\rm max}\{a,b \}$ and $a \neq b$

(v) ${\rm min}\{a,b \}<p-(a+b)<{\rm max}\{a,b \}.$

Let us see that these are the only possible cases. Suppose first that  $a=b.$ In this case, if $a<p-2a,$ then (i) holds; if $p-2a<a$ then (iii) holds. Note that since, by assumption, $p$ is a prime greater than 3, $p-2a \neq a.$ Next suppose that $a \neq b.$ Let $m:={\rm min}\{a,b \},$ $M:={\rm max}\{a,b \}$ and $c:=p-(a+b).$ If $M<c$ then (i) holds; if $c=M$ then, (iv) holds; if $m<c<M,$ then (v) holds; if $c\leq m,$ then (ii) holds. Let us look at these cases separately: 

(i) In this case, 
$$
\Omega ^{(p)}(\mathfrak{q}'-\mathfrak{q})=-\frac{1}{q}(\frac{x}{s})^q(a \alpha d\beta -b\beta d\alpha )
$$
and 
$$
res_{s=0}\Omega ^{(p)}(\mathfrak{q}'-\mathfrak{q})=-\frac{x^q}{q}\sum_{i+j=q} (aj-bi)\alpha_i \beta _{j}. 
$$

Since $qw=p-(a+b),$  $-\frac{1}{q}(aj-bi)=-\frac{1}{q}(aj-b(q-j))=-\frac{a+b}{q}j+b=wj+b$ in characteristic $p.$ Therefore in this case the residue matches with the difference of the $\ell^{(p)}$'s.
 
(ii) In this case because of the anti-symmetry between the terms we may assume without loss of generality that $a>b.$ Then 

$$
\Omega ^{(p)}(\mathfrak{q}'-\mathfrak{q})=\Omega ^{(p)}(\underline{e}^{\alpha t^a}\wedge \underline{e}^{\beta t^b}\wedge \underline{e}^{-\sum _{1\leq n<p}\frac{1}{n}(\frac{xt^w}{s})^n})
=\alpha (b \beta d(-\frac{1}{q}\frac{x^q}{s^q}) +w\frac{x^q}{s^q} d\beta)=x^q\alpha(\frac{b\beta ds}{s^{q+1}}+\frac{wd\beta}{s^q} ).$$
Hence the residue is
$$
x^q\sum _{i+j=q}(b+wj)\alpha_i\beta _j,
$$
 which again matches exactly to the above difference of $\ell^{(p)}$'s. 
 
(iii) In this case 
$$
\Omega ^{(p)}(\mathfrak{q}'-\mathfrak{q})=\Omega ^{(p)}(\underline{e}^{\alpha t^a}\wedge \underline{e}^{\beta t^b}\wedge \underline{e}^{-\sum _{1\leq n<p}\frac{1}{n}(\frac{xt^w}{s})^n})
=-\frac{1}{q}(\frac{x}{s})^q(a \alpha d\beta -b\beta d\alpha )
$$
and we proceed exactly as in (i).

(iv) In this case again, without loss of generality, assume that $a<b.$ Then we have $$
\Omega ^{(p)}(\mathfrak{q}'-\mathfrak{q})=\Omega ^{(p)}( \underline{e}^{\beta t^b}\wedge \underline{e}^{-\sum _{1\leq n<p}\frac{1}{n}(\frac{xt^w}{s})^n}\wedge \underline{e}^{\alpha t^a})
=\alpha (b \beta (-\frac{1}{q}d\frac{x^q}{s^q})  +w\frac{x^q}{s^q} d\beta)
$$ 
and we proceed as in (ii).

(v) In this case, again assuming without loss of generality that $a<b,$ we proceed exactly as in (iv). This finishes the proof of the lemma.
\end{proof}

We will need the following very elementary observation for the proof of the general proposition on the residues. 

\begin{lemma}\label{lemma-change-unif}
Suppose that $s' \in k[[s,t]]$ with  $s'|_{t^2}=s|_{t^2}.$ There exist $u \in k[[s,t]]^{\times}$ and $x_{w}\in k,$ for $w\geq 2$ such that $s'=us+\sum_{2\leq w}x_wt^w.$
\end{lemma}

\begin{proof}
By assumption, $s'=s+t^2(\alpha +s\beta)$ with $\alpha \in k[[t]]$ and $\beta \in k[[s,t]].$ If we put $u=1+t^2\beta$ and $t^2\alpha =\sum _{2\leq w}x_w t^w,$ we obtain the expression we are looking for. 
\end{proof}

\begin{lemma}\label{residue-different-unif}
Suppose that $s'|_{t^2}=s|_{t^2}$ and that $
\mathfrak{q}'=s'\wedge \mathfrak{p} 
$
and 
$
\mathfrak{q}=s\wedge \mathfrak{p}
,$ where $\mathfrak{p}=\underline{e}^{\alpha t^a}\wedge \underline{e}^{\beta t^b},$
with $0 \leq a,b$ and $\alpha, \beta \in k[[s]],$ then we have
\begin{align}
 \ell^{(p)}(res_{s'}(\mathfrak{q}'))-\ell^{(p)} (res_{s}(\mathfrak{q}))=res_{s=0} \Omega ^{(p)}(\mathfrak{q}'-\mathfrak{q}).
\end{align}
\end{lemma}

\begin{proof}
Let us write $s'=us +\sum_{2\leq w}x_wt^w$ as in Lemma \ref{lemma-change-unif} above.  If we let $s'':=us$ and $\mathfrak{q}'':=s''\wedge \mathfrak{p}$ then  $res_{s''}( \mathfrak{q}'')=res_{s}(\mathfrak{q})$ and since $\Omega ^{(p)}(\mathfrak{q}''-\mathfrak{q})=\Omega ^{(p)}(u\wedge \mathfrak{p}) \in \Omega ^{1} _{k[[s]]/k},$ $res_{s=0}\Omega ^{(p)}(\mathfrak{q}''-\mathfrak{q})=0.$

Let us put $s_1=s''$ and $s_{i+1}=s_i+x_{i+1}t^{i+1}$ for $1\leq i<p-1.$ If we also let $\mathfrak{q}_i:=s_{i}\wedge \mathfrak{p}$ and apply Lemma \ref{lemma-special-residue-formula} to the pair $\mathfrak{q}_{i+1}$ and $\mathfrak{q}_{i}$ for all $1 \leq i <p-1$ and take their sum, we obtain 
\begin{align*}
 \ell^{(p)}(res_{s_{p-1}}(\mathfrak{q}_{p-1}))-\ell^{(p)} (res_{s_1}(\mathfrak{q}_1))=res_{s=0} \Omega ^{(p)}(\mathfrak{q}_{p-1}-\mathfrak{q}_1).
\end{align*}
Using the previous paragraph, this can be rewritten as: 
\begin{align*}
 \ell^{(p)}(res_{s_{p-1}}(\mathfrak{q}_{p-1}))-\ell^{(p)} (res_{s}(\mathfrak{q}))=res_{s=0} \Omega ^{(p)}(\mathfrak{q}_{p-1}-\mathfrak{q}).
\end{align*}

In order finish the proof of the lemma, we need to compare the terms corresponding to $s'$ and $s_{p-1}.$ Note that $s'-s_{p-1}=t^{p}f(t),$ with $f(t) \in k[[t]].$ This implies that the images of  $res_{s'}(s'\wedge \mathfrak{p})$ and $res_{s_{p-1}}(s_{p-1}\wedge \mathfrak{p})$ are  equal to each other in $\Lambda ^{2}k_{p} ^{\times}$ and hence that 
$\ell^{(p)}(res_{s'}(\mathfrak{q}'))=\ell^{(p)}(res_{s_{p-1}}(\mathfrak{q}_{p-1})).$ 
Since $\mathfrak{q'}-\mathfrak{q}_{p-1}=(1+\frac{t^pf(f)}{s_{p-1}})\wedge \mathfrak{p}, $ $\mathfrak{q'}-\mathfrak{q}_{p-1}$ has image 0 in $\Lambda ^{3}k((s))_{p}$ and hence $\Omega ^{(p)}(\mathfrak{q'}-\mathfrak{q}_{p-1})=0.$ This completes the proof. 
\end{proof}

\begin{lemma}\label{lemma-same-unif}
Suppose that $a,b\geq 0,$  and $\alpha, \, \beta \in k[[s]],
$   then 
$$
\ell^{(p)}(res_{s}(s\wedge \underline{e}^{\alpha t^a} \wedge \underline{e}^{\beta t^b} ))=res_{s=0} \Omega ^{(p)}(s \wedge \underline{e}^{\alpha t^a} \wedge \underline{e}^{\beta t^b} )
$$
\end{lemma}

\begin{proof}
Both sides of the expression above are equal to 0 if $a+b\neq p,$ or if $a=0$ or $b=0.$ Suppose that $a+b=p,$ and $a, b >0.$  Without loss of generality assume that  $a>b.$ Then the right hand side is 
$
res_{s=0}(\alpha b\beta \frac{ds}{s})=b\alpha(0)\beta(0).
$ The left hand side is $\ell _{p}(\underline{e}^{\alpha (0)t^a}\wedge \underline{e}^{\beta(0)t^b})=b\alpha (0) \beta (0).$ This proves the lemma.
\end{proof}

\noindent We will use the lemmas above to prove the following. 

\begin{proposition}\label{prop-res-add}
Suppose that $s|_{t^2}=s'|_{t^2}$ and $\mathfrak{q}$ is $s$-good and $\mathfrak{q}'$ is $s'$-good and $\mathfrak{q}|_{t^2}=\mathfrak{q}'|_{t^2}.$ Then we have 
\begin{align}\label{res-formula-inside-prop}
 \ell^{(p)}(res_{s'}(\mathfrak{q}'))-\ell^{(p)} (res_{s}(\mathfrak{q}))=res_{s=0} \Omega ^{(p)}(\mathfrak{q}'-\mathfrak{q}).
\end{align}
 \end{proposition}

\begin{proof} First, we note that when $\mathfrak{q} \in \Lambda ^{3}k[[s]]_{\infty} ^{\times}$ then the same is true for $\mathfrak{q}'.$ This shows that the left hand side of the equation is 0 since $res_{s'}(\mathfrak{q}')=res_{s}(\mathfrak{q})=0.$ Similarly, the right hand side of the equation is 0 as well since $\Omega ^{(p)}(\mathfrak{q}'-\mathfrak{q}) \in \Omega ^{1} _{k[[s]]/k}.$ 

If $\tilde{\alpha}\in k[[s,t]]$ then the reduction of $s^{n}\underline{e}^{\tilde{\alpha}} $ modulo $(t)$ is an element of the discrete valuation ring $k((s)),$ with valuation $n.$ The same is true for $(s')^{n}\underline{e}^{\tilde{\alpha}},$ since by assumption $s'$ also reduces to $s$ in $k((s)).$  This implies that, if $(s^{n}\underline{e}^{\tilde{\alpha}})|_{t^2} =((s')^{m}\underline{e}^{\hat{\alpha}})|_{t^2}$ then $n=m$ and since by assumption $s|_{t^2}=s'|_{t^2}$ we also have $\tilde{\alpha}|_{t^2}=\hat{\alpha}|_{t^2}.$  

Therefore, if we start with $\mathfrak{q}$ and ${\mathfrak{q}}' $ as in the statement of the proposition then 
$$
\mathfrak{q}=(s^{n_1}\underline{e}^{\tilde{\alpha}})\wedge (s^{n_2}\underline{e}^{\tilde{\beta}}) \wedge (s^{n_3}\underline{e}^{\tilde{\gamma}})
$$
and 
$$
\mathfrak{q}'=((s')^{n_1}\underline{e}^{\hat{\alpha}})\wedge ((s')^{n_2}\underline{e}^{\hat{\beta}}) \wedge ((s')^{n_3}\underline{e}^{\hat{\gamma}}),
$$
with $\tilde{\alpha}|_{t^2}=\hat{\alpha}|_{t^2},$ $\tilde{\beta}|_{t^2}=\hat{\beta}|_{t^2},$ and $\tilde{\gamma}|_{t^2}=\hat{\gamma}|_{t^2}.$ 

We have seen above that the contributions from $\underline{e}^{\tilde{\alpha}}\wedge \underline{e}^{\tilde{\beta}} \wedge \underline{e}^{\tilde{\gamma}} $ and $\underline{e}^{\hat{\alpha}}\wedge \underline{e}^{\hat{\beta}} \wedge \underline{e}^{\hat{\gamma}} $ to both sides of (\ref{res-formula-inside-prop}) are equal to 0. Therefore, we only need to prove the statement  in the case when $\mathfrak{q}'=s'\wedge \underline{e}^{\tilde{\alpha} }\wedge \underline{e}^{\tilde{\beta}}$ and $\mathfrak{q}=s\wedge \underline{e}^{\hat{\alpha} }\wedge \underline{e}^{\hat{\beta}},$ where   $\tilde{\alpha}, \tilde{\beta}, \hat{\alpha}, \hat{\beta} \in k[[s,t]]$ with $\tilde{\alpha}|_{t^2}=\hat{\alpha}|_{t^2} $ and $\tilde{\beta}|_{t^2}=\hat{\beta}|_{t^2} .$ In order to prove this statement, it is enough to prove (\ref{res-formula-inside-prop}) in the following two special cases, with $\alpha, \beta \in k[[s,t]]$: 

(i) $\mathfrak{r}'=s'\wedge \underline{e}^{\alpha}\wedge \underline{e}^{\beta}$ and $\mathfrak{r}=s\wedge \underline{e}^{\alpha}\wedge \underline{e}^{\beta}$ 

and 

(ii) $\mathfrak{r}'=s\wedge \underline{e}^{\alpha} \wedge \underline{e}^{\beta }$  and $\mathfrak{r}=0.$ 

Let us first show how (i) and (ii) implies (\ref{res-formula-inside-prop}) for  $\mathfrak{q}'=s'\wedge \underline{e}^{\tilde{\alpha} }\wedge \underline{e}^{\tilde{\beta}}$ and $\mathfrak{q}=s\wedge \underline{e}^{\hat{\alpha} }\wedge \underline{e}^{\hat{\beta}}$ as above. Using (ii) first for $\mathfrak{p}:=s\wedge  \underline{e}^{\tilde{\alpha} }\wedge \underline{e}^{\tilde{\beta}}$ and 0 and then for $\mathfrak{q}=s\wedge \underline{e}^{\hat{\alpha}}\wedge \underline{e}^{\hat{\beta}}$ and 0,  we obtain that: 
\begin{align}\label{local-eq1}
 \ell^{(p)}(res_{s}(\mathfrak{p}))=res_{s=0} \Omega ^{(p)}(\mathfrak{p}).
\end{align}
and 
\begin{align}\label{local-eq2}
 \ell^{(p)}(res_{s}(\mathfrak{q}))=res_{s=0} \Omega ^{(p)}(\mathfrak{q}).
\end{align}
Now applying (i) with $\mathfrak{r}'=\mathfrak{q}'$ and $\mathfrak{r}=\mathfrak{p}$ gives: 
\begin{align}\label{local-eq3}
 \ell^{(p)}(res_{s'}(\mathfrak{q}'))-\ell^{(p)} (res_{s}(\mathfrak{p}))=res_{s=0} \Omega ^{(p)}(\mathfrak{q}'-\mathfrak{p}).
\end{align}
Combining (\ref{local-eq1}), (\ref{local-eq2}) and (\ref{local-eq3}) gives the expression we were looking for. 

Finally, we remark that (i) follows from Lemma \ref{residue-different-unif}, and  (ii) follows from Lemma \ref{lemma-same-unif}. This finishes the proof. 
\end{proof}

\begin{remark} Note that in Lemma \ref{exact-lemma}, where we show that the residue of $\Omega^{(p)}$ is independent of the coordinates, we only need $w\geq 1.$ Therefore, the residues of $\Omega^{(p)}$ are independent of the chosen coordinates when the choice of coordinates  agree  modulo $(t).$ Recall that, in the characteristic 0 case, which was discussed in \cite{un-inf} and \cite{un-higher}, the condition that    $w \geq 2$ was essential. This might mistakenly suggest the reader that in characteristic $p,$ one might define the Chow-Kontsevich dilogarithm for a curve over $k$ rather than  for a curve over $k_{2}.$ In fact, this  is not true. The argument of comparing  the residues of $\Omega^{(p)}$ to the values of $\ell i^{(p)} _{2}$  does not work if we only consider matching of the parameters only in modulo $(t).$  More precisely, we need $w \geq 2$ in Lemma \ref{lemma-special-residue-formula}. The following example shows that this lemma does not hold with $w=1.$  
 
Let us consider the ring  $k[[s,t]]$ and choose the elements
$$
\mathfrak{q}'=(s-t)\wedge (1+s^{p-1})\wedge (1+s)
$$
and 
$$
\mathfrak{q}=s\wedge (1+s^{p-1})\wedge (1+s).
$$
These two elements are the same modulo $(t)$ but not modulo $(t^2).$ 
Note that 
$$
res_{s-t}\mathfrak{q}'=(1+t^{p-1}) \wedge (1+t)
$$
and since $\ell^{(p)} ((1+t^{p-1}) \wedge (1+t))=1$ we have $\ell^{(p)} (res_{s-t}\mathfrak{q}')=1.$ On the other hand, $res_{s}(\mathfrak{q})=0.$ So we have, 
$$
\ell^{(p)} (res_{s-t}\mathfrak{q}')-\ell^{(p)} (res_{s}\mathfrak{q})=1.
 $$
 On the other hand, 
 $$
 \Omega^{(p)}(\mathfrak{q}'-\mathfrak{q})=\Omega^{(p)} ((1-\frac{t}{s})\wedge (1+s^{p-1})\wedge (1+s))=0.
 $$
 Therefore, 
 $$
 \ell^{(p)} (res_{s-t}\mathfrak{q}')-\ell^{(p)} (res_{s}\mathfrak{q})=1\neq 0=res_{s=0}\Omega^{(p)} (\mathfrak{q}'-\mathfrak{q}).
 $$
 \hfill $\Box$
\end{remark}

Suppose that $\pazocal{R}$ and $\pazocal{R}'$ are  smooth $k_{p}$-algebras of relative dimension 1, together with a $k_{2}$-isomorphism $\chi:\pazocal{R} /(t^2) \to \pazocal{R}' /(t^2)$ of their reductions modulo $(t^2).$   We identify the spectra of $\pazocal{R}$ and $\pazocal{R}'$ through this isomorphism and assume that   $c$ is a closed point and $\eta$ is the generic point of this spectrum. Let us  assume further that $\mathfrak{c}$ and $\mathfrak{c}'$  are smooth liftings of $c$ to $\pazocal{R}$ and $\pazocal{R}'$ and that the reductions of $\mathfrak{c}$ and $\mathfrak{c}'$ modulo $(t^2)$ map to each other under $\chi.$    Finally, suppose that $\mathfrak{q} \in \Lambda ^{3}\pazocal{R}_{\eta} ^{\times}$ is $\mathfrak{c}$-good and $\mathfrak{q}' \in \Lambda ^{3}(\pazocal{R}'_{\eta} )^{\times}$ is   $\mathfrak{c}'$-good and that their reductions  modulo $(t^2)$ map to each other under $\chi_{\eta}.$  
 
 Similar to  the discussion following  (\ref{main-complexofsheaves-inf}), there is a unique isomorphism between the $k_{p}$-algebras $k(\mathfrak{c})$ and $k(c)_{p}$ which is the identity map modulo $(t).$ Using this, we identify $\Lambda^{2} k(\mathfrak{c})^{\times} $ with $\Lambda^{2}k(c)_{p}^{\times}.$ Combining this identification with the map $\ell^{(p)}: \Lambda^{2} k(c)_p^{\times} \to k(c) ,$ gives us a map  
$$
\ell ^{(p)}:\Lambda^{2} k(\mathfrak{c})^{\times} \to k(c),
$$
which we denote by the same symbol. 

\begin{corollary}\label{cor-res-form}
With the notation as above, we have 
$$
res_{c}\omega^{(p)}(\mathfrak{q},\mathfrak{q}')=\ell ^{(p)} (res_{\mathfrak{c}}(\mathfrak{q}))-\ell ^{(p)} (res_{\mathfrak{c}'}(\mathfrak{q}'))
$$
\end{corollary}

\begin{proof}
This is a restatement of  Proposition \ref{prop-res-add}.
\end{proof}

\section{An invariant of cycles in characteristic $p$}

We will use the above  constructions 
in order  to define the  Chow-Kontsevich dilogarithm and in turn use this dilogarithm to construct an infinitesimal invariant of one dimensional cycles in three dimensional space over a field of characteristic $p.$ 

\subsection{Proof of Theorem \ref {theorem-main1}} We separate the proof to two parts. In the first part we give the construction of the Chow-Kontsevich dilogarithm. In the second part, we compute this dilogarithm on the projective line. 

\subsubsection{Construction of the Chow-Kontsevich dilogarithm}\label{section-Construction of the Chow-Kontsevich dilogarithm} 
Suppose that $C$ is a smooth and proper curve over $k_{2}$ where $k$ is a field of characteristic $p >3.$ We fix a smooth lifting $\mathfrak{c}$ for each closed point $c$ of $C$ and we let $\mathcal{P}$ be the set of all these smooth liftings as in \textsection \ref{subsubsection-inf-chow}.   We saw in \textsection \ref{section-mod-in-char} that we have a map 
$$
\rho: \Lambda^{3}k(C,\mathcal{P})^{\times} \to k,
$$
which is essentially the same one that was constructed in \cite{un-inf} for the characteristic 0 case.

We will construct another such map 
$$
\rho_{K}: \Lambda ^{3}k(C,\mathcal{P})^{\times} \to k
$$
that is based on the constructions of this paper and has a distinctly  characteristic $p$ flavor.  We refer to \textsection \ref{subsubsection-inf-chow} for the details of the construction of $\rho$ and the notation.  We will  follow the same notation for the construction of $\rho_{K}.$

 Suppose that $f, g, h$ are in $k(C,\mathcal{P})^{\times}.$ Denote the triple of functions $(f,g,h)$ by $p.$ Suppose that we fix:

(i)   $\tilde{\pazocal{A}},$  a smooth $k_{\infty}$-algebra together with  an isomorphism 
$
\alpha: \tilde{\pazocal{A}}/(t^2)\xrightarrow{\sim} \pazocal{O}_{C,\eta},
 $
 
 (ii)  $\tilde{p}_{\eta},$  a triple of functions in $\tilde{\pazocal{A}},$ whose reductions modulo $(t^2)$ map to  $p_{\eta}$ via $\alpha.$ Here $p_{\eta}$ denotes the triple of functions, which are the images in $\pazocal{O}_{C,\eta}$ of the functions in $p.$  
 
 (iii) $\widetilde{\pazocal{B}}^{\circ}_{c},$   a smooth $k_{\infty}$-algebra together with an isomorphism $
\tilde{\gamma}_{c}: \widetilde{\pazocal{B}}^{\circ}_{c}/(t^2) \xrightarrow{\sim}  \hat{\pazocal{O}}_{C,c},
$ 
 for each $c \in |C|,$
 
 (iv) $\tilde{q}_{c},$ a triple of elements in  the localization of  $\widetilde{\pazocal{B}}^{\circ}_{c}$ at its minimal prime ideal (generic point),  which give a good lifting of the image of $p_{\eta}$ under the map $\tilde{\gamma}_{c,\eta}^{-1},$ for each  $c \in |C|.$   Here $\tilde{\gamma}_{c,\eta}^{-1} $ denotes the localization of the inverse of $\tilde{\gamma}_{c}$ at the generic point  $\eta.$

Let  $\alpha_c$ denote the  completion of  $\alpha$ at $c.$ The composition  $ \tilde{\gamma}_{c,\eta} ^{-1} \circ \alpha_c $ identifies  $\hat{\tilde{\pazocal{A}}}_{c}/(t^2)$  with  $\widetilde{\pazocal{B}}^{\circ}_{c,\eta}/(t^2).$ Denote the image of the triple of functions $\tilde{p} _{\eta}$ in the completion $\hat{\tilde{A}}_{c}$ at $c$ of $\tilde{\pazocal{A}}$ with the  symbol  $\tilde{p} _{\eta,c}.$ Similarly, denote 
the image of the triple of functions $\tilde{q}_{c}$ in the localization $\widetilde{\pazocal{B}}^{\circ}_{c,\eta}$ of  $\widetilde{\pazocal{B}}^{\circ}_{c}$ at its generic point, by $\tilde{q}_{c,\eta}.$ The reduction of $\tilde{p} _{\eta,c}$ modulo $(t^2)$ is mapped to the reduction of  $\tilde{q}_{c,\eta}$ modulo $(t^2)$ by the map $ \tilde{\gamma}_{c,\eta} ^{-1} \circ \alpha_c .$ In precisely this situation, we know that the residue $res_{c}  \omega^{(p)}(\tilde{p}_{\eta,c},\tilde{q}_{c,\eta} ,  \tilde{\gamma}_{c,\eta} ^{-1} \circ \alpha_c ))$ of $\omega^{(p)}$ at $c$ associated to $\tilde{p} _{\eta,c},$  $\tilde{q}_{c,\eta}$ and the map $\tilde{\gamma}_{c,\eta} ^{-1} \circ \alpha_c $  is well-defined by \textsection \ref{section definition omega}.

On the other hand, using the  notation   before  (\ref{ord-chow-dilog-eq}), $res_c(\tilde{q}_{c}) \in \Lambda ^{2}k(\tilde{\mathfrak{c}})^{\times}.$ Similar to the notation in Corollary \ref{prop-res-add}, by identifying the $k_{\infty}$-algebras $k(\tilde{\mathfrak{c}})$ and $k(c)_{\infty}$ with the unique map which is  identity  modulo $(t), $ we have a canonical map
$$
\ell ^{(p)}:\Lambda^{2} k(\tilde{\mathfrak{c}})^{\times} \to k(c).
$$
 Therefore, we can  define $\ell ^{(p)}(res_c(\tilde{q}_{c})) \in k(c).$

The value of $\rho_{K}$ on $p$ is defined by:

\begin{eqnarray}\label{main-expression}
\rho_{K}(p):= \sum_{c\in |C|}{\rm Tr}_k(\ell^{(p)}(res_c(\tilde{q}_{c}) )+ res_{c}  \omega^{(p)}(\tilde{p}_{\eta,c},\tilde{q}_{c,\eta} ,  \tilde{\gamma}_{c,\eta} ^{-1} \circ \alpha_c )).
\end{eqnarray}

\noindent We need to check that the above sum is finite and is independent of the auxiliary choices. 

\subsubsection{The finiteness of (\ref{main-expression})}

Let $S \subseteq |C|$ be a finite subset of the set of  closed points of $C,$ outside of which  all the functions in $p$ and $\tilde{p}_{\eta}$ are regular. Let $c \in |C| \setminus S.$ 

Since $p$ is regular at $c,$ so is $\tilde{q}_{c}$ and hence $res_c(\tilde{q}_{c})=0.$ Similarly, since  $\tilde{p}_{\eta}$ is regular at $c,$ $res_{c}(\tilde{p}_{\eta})=0.$  These imply that $\ell^{(p)}(res_c(\tilde{q}_{c}) )=\ell^{(p)}(res_c(\tilde{p}_{\eta}) )=0.$ From Corollary \ref{cor-res-form}, we also deduce that 
$$
res_{c}  \omega^{(p)}(\tilde{p}_{\eta,c},\tilde{q}_{c,\eta} ,  \tilde{\gamma}_{c,\eta} ^{-1} \circ \alpha_c )=\ell^{(p)}(res_c(\tilde{p}_{\eta}) )-
\ell^{(p)}(res_c(\tilde{q}_{c}) )=0.$$ 
This implies that the sum in (\ref{main-expression}) is in fact over $S$ and hence is finite.

\subsubsection{Independence of (\ref{main-expression}) from the choices} 

Fix a $c \in |C|.$ 
Let $\widetilde{\pazocal{B}}'^{\circ}_{c},$  be another smooth $k_{\infty}$-algebra together with an isomorphism $
\tilde{\gamma}'_{c}: \widetilde{\pazocal{B}}'^{\circ}_{c}/(t^2) \xrightarrow{\sim}  \hat{\pazocal{O}}_{C,c}
$ and $\tilde{q}'_{c} ,$ a triple of elements in  the localization of  $\widetilde{\pazocal{B}}'^{\circ}_{c}$ at its minimal prime ideal   which gives a good lifting of the image of $p_{\eta}$ under the inverse of the map $\tilde{\gamma}'_{c,\eta}.$

The difference of 
$$
\ell^{(p)}(res_c(\tilde{q}_{c}) )+ res_{c}  \omega^{(p)}(\tilde{p}_{\eta,c},\tilde{q}_{c,\eta} ,  \tilde{\gamma}_{c,\eta} ^{-1} \circ \alpha_c )
$$
and 
$$
\ell^{(p)}(res_c(\tilde{q}'_{c}) )+ res_{c}  \omega^{(p)}(\tilde{p}_{\eta,c},\tilde{q}'_{c,\eta} ,  (\tilde{\gamma}'_{c,\eta} )^{-1} \circ \alpha_c )
$$
is equal to 
$$
\ell^{(p)}(res_c(\tilde{q}_{c}) )-
\ell^{(p)}(res_c(\tilde{q}'_{c}) )
-res_{c}  \omega^{(p)}(\tilde{q}_{c,\eta},\tilde{q}'_{c,\eta} ,  (\tilde{\gamma}'_{c,\eta} )^{-1} \circ \tilde{\gamma}_{c,\eta}  )
$$
by Lemma \ref{lemma-sum-residue}. On the other hand, the last expression is 0 by Corollary \ref{cor-res-form}. This proves independence from the local choices. 

Next we prove  independence from the global choice.  Let $\tilde{\pazocal{A}}'$ be    a smooth $k_{\infty}$-algebra together with  an isomorphism 
$
\alpha': \tilde{\pazocal{A}}'/(t^2)\xrightarrow{\sim} \pazocal{O}_{C,\eta}
 $ and 
  $\tilde{p}_{\eta}',$  be a triple of functions in $\tilde{\pazocal{A}}',$ whose reductions modulo $(t^2)$ map to  $p_{\eta}$ via $\alpha'.$ Similar to above, the difference of  
  $$
\ell^{(p)}(res_c(\tilde{q}_{c}) )+ res_{c}  \omega^{(p)}(\tilde{p}_{\eta,c},\tilde{q}_{c,\eta} ,  \tilde{\gamma}_{c,\eta} ^{-1} \circ \alpha_c )
$$
and 
$$
\ell^{(p)}(res_c(\tilde{q}_{c}) )+ res_{c}  \omega^{(p)}(\tilde{p}'_{\eta,c},\tilde{q}_{c,\eta} ,  \tilde{\gamma}_{c,\eta} ^{-1} \circ \alpha_c ')
$$
is equal to 
$$
res_{c}  \omega^{(p)}(\tilde{p}_{\eta,c},\tilde{p}'_{\eta,c} ,  (\alpha' _{c})^{-1}\circ \alpha_c  )
$$
by Lemma \ref{lemma-sum-residue}. This implies that the difference between the expressions (\ref{main-expression}) corresponding to $\tilde{p}_{\eta}$ and $\tilde{p}'_{\eta}$ is equal to 
$$
\sum _{c \in |C|} {\rm Tr} _{k}(res_{c}  \omega^{(p)}(\tilde{p}_{\eta,c},\tilde{p}'_{\eta,c} ,  (\alpha' _{c})^{-1}\circ \alpha_c  )).
$$
Let us  choose an isomorphism $\beta: \tilde{\pazocal{A}} \to \tilde{\pazocal{A}}'$ whose reduction modulo $(t)$ is the same as the reduction of $(\alpha ')^{-1} \circ \alpha,$ and $\varphi: k(C)_{\infty} \to  \tilde{\pazocal{A}}_{\eta}$ is any $k$-algebra isomorphism. Let $\tilde{q}$ denote the inverse image of $\tilde{p}_{\eta}$ under $\varphi$ and $\tilde{q}'$ denote the inverse image of $\tilde{p}_{\eta} '$ under $\beta_{\eta}\circ\varphi.$ Then $\tilde{q}$ and $\tilde{q}'$ are triples of functions in $k(C)_{\infty}$ with the same reduction modulo $(t).$ 

By its definition, for any $c \in |C|,$  $res_{c}  \omega^{(p)}(\tilde{p}_{\eta,c},\tilde{p}'_{\eta,c} ,  (\alpha' _{c})^{-1}\circ \alpha_c  )$ is equal to the residue of  
$$
\Omega ^{(p)}(\tilde{q}-\tilde{q}') \in \Omega^{1} _{k(C)/k}
$$
at $c. $ Since the sum of the residues of a 1-form is 0, this finishes the proof of the independence.

 \subsubsection{Computation of the Chow-Kontsevich dilogrithm on the projective line} In this example, we assume that $k$ is algebraically closed. As $C$ we take the projective line $\mathbb{P}^1 _{k_2},$ with coordinate function $z.$ 
 We fix a set $\mathcal{P}$ of smooth liftings for each closed point on the projective line and let $f, g, $ and $h \in k(\mathbb{P}^1 _{k_2},\mathcal{P})^{\times}.$

 We assume that for $\infty,$ the standard lifting is in $\mathcal{P}.$  We can reduce to this case using the functoriality of the construction with respect to automorphisms of the projective line. Since we assume that $k$ is algebraically closed, we can write the reduction $\underline{f}$ of $f$ modulo $(t)$ as a product of linear terms in $z.$ For $x \in \mathbb{A}^{1}_k=k,$ let $\tilde{x} \in \mathcal{P}$  denote the unique lifting with reduction $x,$ and let $\nu_{x}(\underline{f})$ denote the valuation of $\underline{f}$ at $x.$ Then, using the assumption that $f$ is $\mathcal{P}$-good,  we can write 
 $$
f=\tilde{\lambda}\prod_{x \in k} (z-\tilde{x})^{\nu_{x}(\underline{f})},
$$
with $\tilde{\lambda} \in k_{2} ^{\times}.$  

Applying the same argument to $g$ and $h,$ and using the multi-linearity of $\rho_K,$ we reduce to the case of computing $\rho_{K}((z-\alpha)\wedge (z-\beta) \wedge (z-\gamma)),$ with $\alpha,\,  \beta,\, \gamma,\, \infty \in \mathcal{P}.$   Using functoriality with respect to the map that sends $z$ to $\frac{z-\beta}{\alpha -\beta}, $ we reduce to computing, 
$
\rho_{K}((1-z)\wedge z \wedge (z-\varepsilon))
$
with $0,\, 1,\, \infty,\,   \varepsilon \in \mathcal{P}'.$ Let us put, 
$\varepsilon=s+as(1-s)\cdot t \in k_{2} ^{\flat},$ with $s \in k^{\flat}, $ $a \in k, $ and   choose $\tilde{\varepsilon}=s+as(1-s)\cdot t \in k_{\infty} ^{\flat}$ as a lifting of $\varepsilon.$ 

Then, in order to do the computation, we may choose the global lifting $\mathbb{P}^{1} _{k_{\infty}}$ and the triple of functions 
 
 $$
 \tilde{p}:=(1-z)\wedge z \wedge (z-\tilde{\varepsilon})
 $$ as a global lift of the functions.  The only singularities of $\tilde{p}$ are at $0,1,\infty$ and $\tilde{\varepsilon}.$ Since $res_{\delta} \tilde{p}=0$ for $\delta \in \{0,1,\infty \},$ and 
 $res_{\tilde{\varepsilon}}=(1-\tilde{\varepsilon})\wedge \tilde{\varepsilon},$
 we find that 
 $$
 \rho_K((1-z)\wedge z \wedge (z-\varepsilon))=\ell ^{(p)}((1-\tilde{\varepsilon}) \wedge \tilde{\varepsilon})=\ell i_{2} ^{(p)}(\varepsilon)=a^{p}\cdot \pounds _{1}(s).
 $$

\subsection{An infinitesimal invariant of cycles in characteristic $p$}

We constructed an invariant of codimension two cycles in the  three dimensional space over a field of characteristic 0 in \cite[\textsection 4]{un-inf} which we denoted by $\rho_{f}.$  In characteristic $p\geq 5,$   this construction carries without any modification and we   denote this regulator by $\rho.$ On the other hand, in characteristic $p,$ there is another regulator which is based on  the Chow-Kontsevich dilogarithm and which we will denote by $\rho_{K}.$ This new regulator has no characteristic 0 analog.  

Since the definitions and the proofs are exactly analogous to those in characteristic 0, we will omit most of the details and refer the reader to \cite[\textsection 4]{un-inf}. We  start by  recalling the definitions of Bloch's higher Chow groups for a smooth scheme over a field and its version
for the truncated polynomial ring over a field from \cite[\textsection 4]{un-inf}.

\subsubsection{Bloch's  higher Chow groups}

 Let $\square_k:= \mathbb{P}^{1} _{k} \setminus \{ 1\}$ and $\square ^n _{k}$ the $n$-fold product of $\square_k$ with itself over $k, $ with the coordinate functions $y_1, \cdots, y_n.$ For  a smooth $k$-scheme $X,$ we let   $\square^n _{X} :=X \times_k \square_k ^n.$  A codimension one face of $\square^n _{X}$ is a divisor $F_{i} ^a$  of the form $y_{i}=a,$ for $1\leq i \leq n,$ and $a \in \{0,\infty \}.$ A face of $\square^n _{X}$ is either the whole scheme $\square^n _{X}$ or an arbitrary intersection of codimension one faces.

 Let $\underline{z}^q (X, n)$ be the free abelian group on the set of codimension $q,$ integral, closed subschemes $Z \subseteq  \square^n _{X}$ which are {\it admissable}, i.e.  which intersect each face properly on $\square^n _{X}.$ For each codimension one face $F_{i} ^a,$ and  irreducible $Z \in \underline{z} ^q (X, n)$, we let $\partial_i ^{a} (Z)$  be the cycle associated to the scheme $Z \cap F_{i} ^{a}.$ We let $\partial:= \sum_{i=1} ^n (-1)^n (\partial_i ^{\infty} - \partial_i ^0)$ on $\underline{z}^q (X, n)$.

 Let  $\underline{z}^q (X, n)_{\rm degn}$ denote the subgroup of  degenerate cycles and 
 $z^q (X, \cdot):= \underline{z}^q (X, \cdot)/ \underline{z}^q (X, \cdot)_{\rm degn}$ the corresponding non-degenerate complex. The complex $(z^q (X, \cdot), \partial)$  is called the \emph{higher Chow complex} of $X$ and its homology ${\rm CH}^q (X, n):= {\rm H}_n (z^q (X, \cdot))$ is the higher Chow group of $X$.

\subsubsection{Cycles over $k_{\infty}$}  

 Let $\overline{\square}_{k}:= \mathbb{P}^{1} _{k},$  $\overline{\square}_{k} ^{n},$ the $n$-fold product of  $\overline{\square}_{k} $ with itself over $k,$ and  $\overline{\square}_{k_{\infty}} ^{n} :=\overline{\square}_{k} ^{n} \times _k k_{\infty}.$ We define a subcomplex $\underline{z}^q _{f} (k_{\infty}, \cdot) \subseteq \underline{z}^q (k_{\infty}, \cdot)$, as  the subgroup generated by integral, closed subschemes $Z \subseteq \square_{k_{\infty}} ^n$ which are admissible and have the property that  $\overline{Z}$  intersects each $s\times \overline{F}$ properly on $\overline{\square}_{k_{\infty}} ^n,$  for  every face $F$ of $\square^n_{k}.$ Here $s$ denotes the closed point of ${\rm Spec}\, k_{\infty},$ and $\overline{Z}$ ( resp. $\overline{F}$) the closure of $Z$ (resp. $F$) in $\overline{\square}_{k_{\infty}} ^n.$    We refer to  such cycles as cycles having finite reduction.   Modding out by degenerate cycles, we  have a complex $z^q_{f} (k_{\infty}, \cdot).$  We expect that this complex to compute the motivic cohomology of $k_{\infty}$ with coefficients in $\mathbb{Z}(q).$

\subsubsection{Definition of the invariant} Let $\eta$ denote the generic point of ${\rm Spec} \, k_{\infty}.$
An irreducible cycle $q$ in $ \underline{z}_{f} ^2 (k_{\infty},2)$ is  determined by its generic   point $q_{\eta} $ of  $\square ^{2} _{\eta},$ such that its   closure $\overline{q}$  in $\overline{\square} ^{2} _{k_{\infty}}$ does not meet $(\{ 0,\infty\} \times \overline{\square}_{k_{\infty}}) \cup ( \overline{\square}_{k_{\infty}} \times \{0, \infty \}).$ Let $\tilde{q}$ denote the normalisation of $\overline{q} $ and $|\tilde{q}_s|$ denote the underlying set of  the  closed fiber $\tilde{q} \times_{k_{\infty}}s$ of $\tilde{q}.$ Let  $\pi: \tilde{q} \to (\overline{\square}  _{k_{\infty}} \setminus \{0,\infty \})^{2}$ denote the composition of the normalization map from $\tilde{q}$ to $\overline{q}$ and the inclusion of $\overline{q}$  to $(\overline{\square}  _{k_{\infty}} \setminus \{0,\infty \})^{2}.$

For $r \in |\tilde{q}_s|,$ let $k(r)$ denote the residue field of $r.$  We have  a surjection $\hat{\pazocal{O}}_{\tilde{q},r} \to k(r).$ Since $k(r)/k$ is finite \'{e}tale there is a unique splitting $k(r)\to \hat{\pazocal{O}}_{\tilde{q},r}$ of the above surjection, and a unique isomorphism $k(r)_{\infty} \xrightarrow{\sim} \hat{\pazocal{O}}_{\tilde{q},r}$ of $k_{\infty}$-algebras which extend this splitting. 

We let $\pi_{r}$ denote the composition of the isomorphism ${\rm Spec} \, k(r)_{\infty} \xrightarrow{\sim} {\rm Spec} \, \hat{\pazocal{O}}_{\tilde{q},r},$ the natural map ${\rm Spec} \, \hat{\pazocal{O}}_{\tilde{q},r} \to \tilde{q} $ and the map $\pi.$ If $y$ is an invertible function on $(\overline{\square}  _{k_{\infty}} \setminus \{0,\infty \})^{2},$ we let 
$$
\ell _{r,i}(y):= \ell _{i} ( \pi_{r} ^{*}(y)), 
$$
for $1\leq i <p.$ Similarly, we let 
$$
\ell(r):=(\ell_{r,2}\wedge \ell_{r,1})(y_{1}\wedge y_2)
$$
and 
$$
\ell ^{(p)}(r):=\frac{1}{2}\sum _{1\leq i <p}i(\ell_{r,p-i}\wedge \ell_{r,i})(y_{1}\wedge y_2),
$$
where $y_{1}$ and $y_{2}$ are the coordinate functions on $(\overline{\square}  _{k_{\infty}} \setminus \{0,\infty \})^{2}.$

We define the values of $\ell $ and $\ell ^{(p)}$ on $q$ by: 
\begin{eqnarray}\label{defnl} 
\;\; \; \ell(q):=\sum _{r \in |\tilde{q}_s|}{\rm Tr}_{k} (\ell (r))
\end{eqnarray}
and 
\begin{eqnarray}\label{defnlp} 
\;\; \; \ell^{(p)}(q):=\sum _{r \in |\tilde{q}_s|}{\rm Tr}_{k} (\ell ^{(p)} (r)).
\end{eqnarray}

%{\color{blue} Note that the multiplicity appears implicitly when the form is pulled back to the normalisation so it does not appear explicitly in the formula above. It would be a good idea to check that the commutativity of the residue  holds with this definition, i.e. when we pass to the normalisation. This is probably what is done in the classical case. }  

\begin{definition}
We   define the regulators
$
\rho$ and  $\rho^{(p)}$ as maps from   $\underline{z}_{f} ^2 (k_{\infty},3)$ to $k$
as the compositions of $\ell$ and $\ell^{(p)}$ with the boundary map: 
$$
\rho:=\ell \circ \partial\;\;\;\;\;\;{\rm and}\;\;\;\;\;\; \rho_K:=\ell^{(p)} \circ \partial.
$$

\end{definition}

 Exactly as in \cite{un-inf}, we see that $\rho$ and $\rho_K$ vanish on boundaries. In other words, the compositions $\rho\circ\partial$ and $\rho_K\circ\partial$ from $\underline{z}^{2} _{f} (k_{\infty},4)$ to $k $ are both 0.

\subsubsection{Modulus property}\label{modulus section} The most important property of the regulator maps $\rho$ and $\rho_{K}$ states, in essence, that these regulators depend only on the reduction of the cycle modulo $(t^2).$ In order to state this property precisely, we will need the following definition.

\begin{definition}
 Suppose that $Z_i \in  \underline{z}^{2} _{f} (k_{\infty},3),$ for $i=1,\,2,$ are irreducible cycles.  We say that $Z_{1}$ and $Z_{2}$ are {\it equivalent modulo $t^m$} if the following condition, which we denote by $(M_{m}),$ holds:
 
 (i) the closure $\overline{Z}_{i}$ of $Z_{i}$ is smooth over $k_{\infty}$ and $(\overline{Z}_i)_{s} \cup (\cup_{j,a} |\partial _j ^{a} Z_i|) $ is a strict normal crossings divisor on $\overline{Z}_i,$ for $i=1, \, 2$

 (ii) $\overline{Z}_{1}|_{t^m}=\overline{Z}_{2} |_{t^m}.$ 
 \end{definition}
 
\begin{remark} 
It would make sense to define the group $ \underline{z}^{2} _{f} (k_{m},3)$ as the quotient of  $\underline{z}^{2} _{f} (k_{\infty},3)$ by the group generated by $Z_{1}-Z_{2},$ for all  $Z_{1}$ and $Z_{2}$  which are equivalent modulo $t^m.$ The next theorem then states that $\rho$ and $\rho_{K}$ induce maps from $\underline{z}^{2} _{f} (k_{2},3)$ to $k.$ We expect that in the future a complex computing the weight two motivic cohomology of $k_2$ will be defined where  $\underline{z}^{2} _{f} (k_{2},3)$ is the degree one term of this complex. We expect $\rho \oplus \rho_K$ to induce an isomorphism from the first cohomology of this complex to $k \oplus k.$ Note that the main result of \cite{un-ao} can be thought of as a {\it linear} version of this statement. Namely, the Bloch complex is a {\it linear} algebraic complex which computes motivic cohomology of weight two and the map $\rho \oplus \rho_K$ restricted to linear cycles as in \textsection  7.1.4 coincides with the map on the Bloch complex, which gives an isomorphism in the cohomology of degree one. 
\end{remark} 

The main result of this section is the following:

\begin{theorem}\label{modulus theorem}
Suppose that  $Z_{i} \in \underline{z} _{f} ^{2}(k_{\infty},3),$ for $i=1,2,$ are two irreducible cycles  which are equivalent modulo $t^2.$ Then we have 
$$\rho (Z_{1})=\rho(Z_{2})$$ 

and 

$$\rho_{K} (Z_{1})=\rho_{K}(Z_{2}).$$ 

\end{theorem}

\begin{proof} 
The proofs of the statements for $\rho$ and $\rho_K$ are essentially the same. Therefore, we will only expound the one for $\rho_K.$ 

Suppose that $Z_{1}$ and $Z_{2}$ are as in the statement. Let us put $\overline{Z}$ to be the common reduction of $\overline{Z}_{i}$ modulo $(t^2).$ Then $\overline{Z}$ is a smooth and projective curve over $k_2$ and $\overline{Z}_{i}$ are its two different liftings. We will denote the restriction of the coordinate function $y_{j}$ to $\overline{Z}$ by $y_{j,\overline{Z}},$ similarly its restriction to $\overline{Z}_{i}$ by $y_{j,\overline{Z}_i}.$   

Let us compute $\rho_{K}(y_{1,\overline{Z}}\wedge y_{2,\overline{Z}} \wedge y_{3,\overline{Z}})$ by using the lifting $\overline{Z}_{i}. $ Note that we proved in the main theorem that the value of $\rho_{K}(y_{1,\overline{Z}}\wedge y_{2,\overline{Z}} \wedge y_{3,\overline{Z}})$ is independent of the choice of the lifting.      

Since $\overline{Z}_i$ and the functions $y_{j,\overline{Z}_{i}}$ are global good  liftings, we can choose $y_{1,\overline{Z}_i}\wedge y_{2,\overline{Z}_i} \wedge y_{3,\overline{Z}_i}$ as both $\tilde{p}_{\eta}$ and $\tilde{q}_{c}$ for each $z \in |\overline{Z}|$ as in the construction in  \textsection \ref{section-Construction of the Chow-Kontsevich dilogarithm}. This would make the defect term  
 $$
res_{z}  \omega^{(p)}(\tilde{p}_{\eta,z},\tilde{q}_{z,\eta} ,  \tilde{\gamma}_{z,\eta} ^{-1} \circ \alpha_z )
 $$
equal to  0 in (\ref{main-expression}) and the expression reduces to 
$$
\rho_{K}(y_{1,\overline{Z}}\wedge y_{2,\overline{Z}} \wedge y_{3,\overline{Z}})=\sum_{z\in |\overline{Z}|}{\rm Tr}_k(\ell^{(p)}(res_z( y_{1,\overline{Z}_i}\wedge y_{2,\overline{Z}_i} \wedge y_{3,\overline{Z}_i}) )=(\ell ^{(p)}\circ \partial)(Z_{i})=\rho_{K}(Z_{i}).
$$
Since the left hand side does not depend on the choice of $Z_{i},$ we obtain $\rho_{K}(Z_{1})=\rho_{K}(Z_{2}).$ 
\end{proof}

%\subsubsection{Vanishing on the products}  Suppose that $Z$ is an irreducible cycle in $\underline{z}_{f} ^2 (S,3)$ which satisfies condition (i) in \textsection \ref{modulus section}.

%\begin{proposition}\label{products} If there is $1\leq i\leq 3$ such that $y_{i} $ restricted to $Z_{2}:=Z|_{t^2}$ is in $k_2 ^{\times}$ then $\rho_f(Z)=0.$ \end{proposition}

%\begin{proof} Without loss of generality assume that $i=1,$ and $y_{1}$ restricted to $Z|_{2}$ is $\alpha \in k_2 ^\times.$ By exactly as in the proof of Theorem \ref{modulus theorem} we see that $$\rho_f(Z)=\rho((y_{1}\wedge y_2 \wedge y_3)|_{Z_2})=\rho(\alpha \wedge( y_2 \wedge y_3)|_{Z_2}).$$On the other hand, the right hand side is 0 by Proposition \ref{rho0}.\end{proof}

{\bf Acknowledgement.} The author thanks the referee for a very careful reading and many suggestions which improved the paper.

\end{document}